\documentclass[11pt]{article}
\usepackage{url,mathrsfs}
\usepackage[american]{babel}
\usepackage{subcaption}
\usepackage{xcolor}
\usepackage{amssymb}
\usepackage{amsfonts}
\usepackage{amsmath}
\usepackage{mathrsfs}
\usepackage{graphicx}
\usepackage{amsbsy}
\usepackage{theorem}
\usepackage{color}
\usepackage{hyperref}
\usepackage{tikz}
\usepackage[normalem]{ulem}
\def\address#1{\expandafter\def\expandafter\@aabuffer\expandafter
	{\@aabuffer{\affiliationfont{#1}}\relax\par
	\vspace*{13pt}}}

\newtheorem{theorem}{Theorem}[section]
\newtheorem{lemma}[theorem]{Lemma}

\newtheorem{proposition}[theorem]{Proposition}
\newtheorem{remark}[theorem]{Remark}

\def\thetheorem{\thesection.\arabic{theorem}}
\def\thesection{\arabic{section}}

\def\theequation {\thesection.\arabic{equation}}
\def\beq{\begin{equation}\displaystyle}
\def\eeq{\end{equation}}
\def\bel{\begin{equation} \displaystyle \begin{array}{l} }
\def\eel{\end{array} \end{equation} }
\def\bell{\begin{equation} \displaystyle \begin{array}{ll}  }
\def\eell{\end{array} \end{equation} }

\def\bea{\begin{eqnarray}}
\def\eea{\end{eqnarray} }
\def\bean{\begin{eqnarray*}}
\def\eean{\end{eqnarray*} }
\newenvironment{proof}{\noindent{\bf Proof.~}}
{{\mbox{}\hfill {\small \fbox{}}\\}}
\catcode`@=11
\renewcommand\appendix{\bigskip {\noindent \Large \bf Appendix}
  \setcounter{section}{0}%
  \setcounter{subsection}{0}%
\setcounter{equation}{0}%
\setcounter{theorem}{0}%
\def\thetheorem{A.\arabic{theorem}}
\def\theequation {A.\arabic{equation}}}
\catcode`@=12

\def\bar#1{{\overline #1}}

\begin{document}
\markboth{N.~Outada, N.~Vauchelet, T.~Akrid and M.~Khaladi}{From Kinetic Theory of Multicellular Systems to Hyperbolic Tissue Equations}


\title{\scshape{From Kinetic Theory of Multicellular Systems to Hyperbolic Tissue Equations:  Asymptotic Limits and Computing}}

\author{Nisrine Outada$^{1,2,3}$, Nicolas Vauchelet$^{2,4}$,\\
Thami Akrid$^{1}$ and Mohamed Khaladi$^{1,3}$\vspace{0.25cm}}

\vskip1cm

\date{$^{1}$D\'epartement de Math\'ematiques, Facult\'e des Sciences Semlalia,\\
Laboratoire LMDP, Universit\'e Cadi Ayyad,\\ B.P. 2390, 40000 Marrakesh, Morocco \\[1mm]
$^{2}$Sorbonne Universit\'es, UPMC Univ Paris 06, UMR 7598,\\
  Laboratoire Jacques-Louis Lions, Paris, France \\[1mm]
$^{3}$UMI 209 UMMISCO, 32 Avenue Henri Varagnat,\\
F-93 143 Bondy Cedex, France \\[1mm]
$^{4}$INRIA-Paris-Rocquencourt, EPC MAMBA, Domaine de Voluceau,\\
BP105, 78153 Le Chesnay Cedex, France 
}

\maketitle

{\bf Dedicated to Abdelghani Bellouquid who prematurely passed away on August 2015.}

\begin{abstract}
This paper deals with the analysis of the asymptotic limit toward the derivation of macroscopic equations for a class of equations modeling complex multicellular systems by methods of the kinetic theory. After having chosen an appropriate scaling of time and space, a Chapman-Enskog expansion is combined with a closed, by minimization, technique to derive hyperbolic models at the macroscopic level. The resulting macroscopic equations show how the macroscopic tissue behavior can be described by hyperbolic systems which seem the most natural in this context. We propose also an asymptotic-preserving well-balanced scheme for the one-dimensional hyperbolic model, in the two dimensional case, we consider a time splitting method between the conservative part and the source term where the conservative equation is approximated by the Lax-Friedrichs scheme.
\end{abstract}

{\bf Keywords.} Kinetic theory; multicellular systems; hyperbolic limits; chemotaxis; asymptotic preserving scheme; Lax-Friedrichs flux.

\section{Introduction}

The aim of this paper is the derivation of  macroscopic hyperbolic models of biological tissues from the underlying  description at the microscopic scale delivered by  kinetic theory methods. We consider the hyperbolic  asymptotic  limit for  microscopic system that connect the biological parameters, at the level of cells,  involved in this level of description.

The first step of the derivation of macroscopic models in biology from the underlying description at the microscopic scale is arguably due to Alt~\cite{alt} and Othmer, Dunbar and Alt~\cite{othmer}, who introduced a new modeling  approach by  perturbation of the transport equation by a velocity jump-process, which appears appropriate to model the velocity dynamics of cells modeled as living particles.   This method has been subsequently developed by various authors, among others, we cite~\cite{bellomo8,bellomo9,bellomo7,bellomo12,bellouquid5,chalub,dolak2,almeida,filbert1,hwang,hillen,hillen2,james,othmer1,perthame,stevens,vauchelet}.
The survey on mathematical challenges on the qualitative and asymptotic analysis of Keller and Segel type models~\cite{bellomo12} reports an exhaustive bibliography concerning different mathematical approach on the aforementioned topics. The interested reader can find a further updating of the research activity on the study of Keller-Segel models and their developments in~\cite{BBC,BT,L,W,WIN}, as well as on applications to population dynamics with diffusion~\cite{TW} and pattern formation in cancer~\cite{SSU}.

Different time-space scalings lead to equations characterized by different parabolic or hyperbolic  structures. Different combinations of parabolic and hyperbolic scales also are used, according to the dispersive or non-dispersive nature of the biological system under consideration. The parabolic (low-field) limit of kinetic equations leads to a drift-diffusion type system (or reaction-diffusion system) in which the diffusion process
dominate the behavior of the solutions\cite{goudon,villani}. On the other hand, in the hyperbolic (high-field) limit the influence of the diffusion terms is of lower (or equal) order of magnitude in comparison with other convective or interaction terms.

Possible applications refer to modeling cell invasion, as well as chemotaxis and haptotaxis phenomena and related  pattern formation\cite{alt,bellomo7,almeida}. Models with finite propagation speed appear to be consistent with physical reality rather than parabolic models. This feature is also induced by the essential characteristics of living organisms who have the ability to sense signals in the environment and adapt their movements accordingly.

Our analysis is  quite general as it can be applied to different species in response to multiple (chemo)tactic cues \cite{corrias,dolak,hillen4,patlak}. Therefore, the derivation of hyperbolic models can contribute to further improvements in modeling biological reality. In fact, it seems that the approach introduced by Patlak~\cite{patlak} and Keller-Segel \cite{keller} is not always sufficiently precise to describe some structures as the evolution of bacteria movements, or the human endothelial cells movements on matrigel that lead to the formation of networks interpreted as the beginning of a vasculature \cite{perthame}. These structures cannot be explained by parabolic models, which generally lead to pointwise blow-up~\cite{bellomo12}, moreover the numerical experiments show the predictability of the hyperbolic models in this context.

We now briefly describe the contents of this paper. Section 2, presents the kinetic model and the scaling deemed to provide the general framework appropriate to derive, by asymptotic analysis, models at the macroscopic scale. Section 3, referring to~\cite{bellomo8}, presents the general kinetic framework to be used toward the asymptotic analysis. Section 4, shows how specific models can be derived by the approach of our paper. Section 5, presents some computational simulations to show the predictive ability of the models derived in this paper and looks ahead to research perspectives.

\section{Kinetic mathematical model}
Let us consider a physical system constituted by a large number of cells interacting in a biological environment. The microscopic state is defined by the mechanical variable $\{x,v\}$, where $\{x,v\}\in \Omega \times V\subset \mathbb{R}^{d}\times \mathbb{R}^{d} $, $d=1,2,3$. The statistical collective description of the system is encoded in the statistical distribution $f = f(t,x,v):[0,T] \times \Omega \times V \rightarrow \mathbb{R}_+$, which is called a distribution function. We also assume that the transport in position is linear with respect to the velocity. In this paper, we are interested in the system of different species in response to multiple chemotactic cues. The model, for $i = 1,\cdots,m,$ reads:
\begin{small}
\begin{eqnarray}\label{SYST}
\begin{cases}
\displaystyle
\vspace{0.25cm}
\partial_t f + v \cdot \nabla_{x}f= L(g,f)+\widetilde{H}(f,g),\\
\displaystyle
\tau_{i} \partial_t g_{i}+ v \cdot \nabla_{x}g_{i}= l_{i}(g_{i})+G_i(f,g),
\end{cases}
\end{eqnarray}\end{small}
where $f=f(t,x,v)$ and $g_i=g_i(t,x,v)$ denotes respectively the density of cells and the density (concentration) of multiple tactic cues and $g=(g_1,\cdots,g_m)^T$.

The operators $L$ and $l_{i}$ model the dynamics of biological organisms by  velocity-jump process. The set of possible velocities is denoted by $V$, assumed to be bounded and radially symmetric. The operators $\widetilde{H}$ and $G_i$ describe proliferation/destruction interactions. The dimensionless time $\tau_i \in \mathbb{R}_+$ indicates that the spatial spread of $f$ and $g_i$ are on different time scales. The case $\tau_i=0$ corresponds to a steady state assumption
 for $g_i$.

The problem of studying the relationships between the various scales of description,  seems to be one of the most important problems of the mathematical modelling of complex systems . Different structures at the macroscopic scale can be obtained corresponding to different spacetime scales. Subsequently, more detailed assumptions on the biological interactions lead to different models of pattern formation. However, a more recent tendency been the use hyperbolic equations to describe intermediate regimes at the macroscopic level rather than parabolic equations, for example \cite{bellomo8,bellomo9,filbert1,hillen2}.

The next section deals with the derivation of macroscopic equations using a Champan-Enskog type perturbation approach for \eqref{SYST}$_1$ and a closure by minimization method for \eqref{SYST}$_2$. Our purpose is to derive hyperbolic-hyperbolic macroscopic model. The first approach consists in expanding the distribution function in terms of a
small dimensionless parameter related to the intermolecular distances (the space scale dimensionless parameter). In \cite{filbert1}, a hydrodynamic limit of such kinetic model was used to derive hyperbolic models for chemosensitive movements. While the closure method consists that the (m+1)-moments of the minimizer approximate the (m+1)-moments of the true solution.

\section{Asymptotic analysis toward derivation of hyperbolic systems}\label{sehyper}
%
\subsection{The kinetic framework}
%
Let us now consider the first equation in \eqref{SYST}. We assume a hyperbolic scaling for this population it means that
we scale time and space variables $t\rightarrow \varepsilon t $ and $x\rightarrow \varepsilon x $, where  $\varepsilon$ is
a small parameter which will be allowed to tend to zero, see \cite{bellomo8} for more details. We deal also with the small interactions i.e
$\widetilde{H}(f,g)=\varepsilon H(f,g)$. Then, we obtain the following transport equation for the distribution function $f=f(t,x,v)$
\begin{equation}\label{eqs}
\partial_t f+ v \cdot \nabla_{x}f= \frac{1}{\varepsilon}L(g,f)+H(f,g) ,
\end{equation}
where the position $x\in \Omega \subset \mathbb{R}^{d}$ and the velocity $v\in V \subset \mathbb{R}^{d}$.
In addition, the analysis developed is based on the assumption that $L$ admits the following decomposition:
\begin{equation}\label{decom}
   L(g,f)=L^0(f)+\varepsilon L^1(g,f),
   \end{equation}
 with $L^1$ in the form
 \begin{equation}\label{expa}
 L^1(g,f)=\sum_{i=1}^{m}L_{i}^1[g_{i}](f).
 \end{equation}
 The operator $L^{0}$ represents the dominant part of the turning kernel modeling the tumble process in the absence of chemical substance
and ${L}_{i}^{1}$ is the perturbation due to chemical cues. The parameter $\varepsilon$ is a time scale which here refers
to the turning frequency.
 The equation \eqref{eqs} becomes
 \begin{equation}\label{eqsc}
\partial_t f+ v \cdot \nabla_{x}f= \frac{1}{\varepsilon}L^{0}(f)+\sum_{i=1}^{m}{L}_{i}^{1}[g_{i}](f)+H(f,g).
\end{equation}
The most commonly used assumption on the perturbation turning operators ${L}^{0}$, $L_i^1$ and $l_i$ is that they are integral operators and read:
\begin{equation}\label{n0}
L^{0}(f)=\int_{V}\left(T^{0}(v,v')f(t,x,v')-T^{0}(v',v)f(t,x,v)\right)dv',
\end{equation}
\begin{equation}\label{n1}
L_{i}^{1}[g_{i}](f)=\int_{V}\left(T_{i}^{1}(g_{i},v,v')f(t,x,v')-T_{i}^{1}(g_{i},v',v)f(t,x,v)\right)dv',
\end{equation}
and
\begin{equation}\label{n2}
l_{i}(f)=\int_{V}\left(K_i(v,v')f(t,x,v')-K_i(v',v)f(t,x,v)\right)dv',
\end{equation}
The turning kernels $ T^{0}(v,v')$, $T_{i}^{1}(g_{i},v,v')$ and $K_i(v,v')$ describe the reorientation of cells, i.e. the random velocity changes from the previous velocity $v'$ to the new $v$.

The following assumptions on the turning operators are needed to develop the hyperbolic asymptotic analysis:\\
$\bullet$ $\textbf{Assumption H0:}$ For all $i=1,\cdots,m$, the turning operators  $L^0$, $L_{i}$ and $l_{i}$ conserve the local mass:
\begin{equation}\label{mass}
\int_{V} L^{0}(f)dv=\int_{V} L_{i}^{1}[g_{i}](f) dv= \int_{V} \l_{i}(f)=0.
\end{equation}
$\bullet$ $\textbf{Assumption H1:}$ The turning operator $L^0$ conserve the population flux:
\begin{equation}\label{flu1}
\int_{V} v L^{0}(f)dv= 0.
\end{equation}
$\bullet$ $\textbf{Assumption H2:}$ For all $ n\in [0,+\infty[$ and $u \in \mathbb{R}^d$, there exists a unique function $F_{n,u} \in L^{1}(V,(1+|v|)dv)$ such that
\begin{equation}\label{stea1}
L^{0}(F_{n,u})=0, \quad \int_{V}F_{n,u} dv=n \quad \text{and} \quad \int_{V} vF_{n,u} dv= nu.
\end{equation}
It is clear, from \eqref{n0}-\eqref{n2}, that $L^0, L_{i}$ and $l_{i}$ satisfy the assumption $\textbf{H0}.$

The following lemma,  whose proof can be found in \cite{hillen4}, will be used a few times,
\begin{lemma}\label{lem}
Assume that $ V= s \mathbb{S}^{d-1}$, $s>0 $, which corresponds to the assumption that any individual of the population chooses any velocity with a fixed norm s(speed). Then,
$$
\int_{V} v dv=0, \quad \int_{V} v^iv^j dv=\frac{|V|s^2}{d}\delta_{ij} \quad \text{and} \quad \int_{V} v^iv^jv^k dv=0,
$$
where $v=(v^1,\cdots,v^d)$ and $\delta_{ij}$ denotes the Kronecker symbol, and the notation $\mathbb{S}^{d-1}$ corresponds to the unit sphere in dimension d.
\end{lemma}
\subsection{Hydrodynamic limit}
In this subsection, we use the last assumptions to derive an hyperbolic system on macroscopic scale for small perturbation parameter.

Let $f$ be solution of the equation \eqref{eqsc} and consider the density of cells $n$ and  the flux $u$ defined by:
\begin{equation}\label{hyd1}
n(t,x)=\int_{V}f(t,x,v) dv \quad \text{and} \quad n(t,x)u(t,x)=\int_{V}v\, f(t,x,v) dv.
\end{equation}
To derive the equations for the moments in \eqref{hyd1}, we multiply \eqref{eqsc} by $1$ and $v$ respectively, and integrate over $V$ to obtain the following system
\begin{eqnarray}\label{hyp1}
\begin{cases}
\vspace{0.25cm}
\displaystyle
\partial_t n+div_x(nu)=\int_{V}H(f,g)dv,\\
\displaystyle
\partial_t(nu)+div_x \int_{V} v\otimes v f(t,x,v) dv=\sum_{i=1}^{m}\int_{V}
 v L_{i}^{1}[g_{i}](f)dv+\int_{V}vH(f,g)dv.
\end{cases}
\end{eqnarray}
Now let $g_{i}$ be a solution of the following (i)-equation
\begin{equation}\label{ieq}
\tau_{i}\frac{\partial g_{i}}{\partial t}+ v \cdot \nabla_{x}g_{i}= l_{i}(g_{i})+G_i(f,g),
\end{equation}
 and set
\begin{equation}\label{hyd2}
N_{i}(t,x)=\int_{V} g_{i}(t,x,v) dv \quad  \text{and} \quad  N_{i}(t,x)U_i(t,x)=\int_{V} v g_{i}(t,x,v) dv.
\end{equation}
To derive the equations for moments in \eqref{hyd2}, we multiply the equation \eqref{ieq} by $1$ and $v$ respectively and integrate over $V$ to obtain the following system
\begin{small}
\begin{eqnarray}\label{syst22}
\begin{cases}
\vspace{0.25cm}
\displaystyle{\tau_{i} \partial_t N_{i} +div_x(N_{i}U_i)=\int_{V} G_i(f,g)dv,}\\
\displaystyle{\tau_{i} \partial_t(N_{i}U_i)+div_x \int_{V} v\otimes v g_{i}(t,x,v) dv= \int_{V} v l_i(g_i)dv+\int_{V} v G_i(f,g)dv.}
\end{cases}
\end{eqnarray}
\end{small}

Finally, \eqref{hyp1} and \eqref{syst22} yield the following system
\begin{small}
\begin{eqnarray}\label{systb}
\begin{cases}
\displaystyle
\partial_t n+div_x(nu)=\int_{V}H(f,g)dv,\\
{}\\
\displaystyle
\partial_t(nu)+div_x \int_{V} v\otimes v f(t,x,v) dv=\sum_{i=1}^{m}\int_{V}
 v L_{i}^{1}[g_{i}](f)dv+\int_{V}vH(f,g)dv,\hspace*{-0.2cm}\\
 {}\\
\displaystyle
\tau_{i} \partial_t N_{i} + div_x(N_{i}U_i)=\int_{V} G_i(f,g)dv,\\
{}\\
\displaystyle
\tau_{i} \partial_t(N_{i}U_i) + div_x \int_{V} v\otimes v g_{i}(t,x,v) dv= \int_{V} v l_i(g_i)dv+\int_{V} v G_i(f,g)dv.\\
\end{cases}
\end{eqnarray}\end{small}
In the following, we are interested to close the system \eqref{systb}. We start by two first equations of \eqref{systb}, we introduce $f_1$ such that
$$
\varepsilon f_1(t,x,v)= f(t,x,v)-F_{n(t,x),u(t,x)}(v),
$$
where the equilibrium distribution $F_{n,u}$ is defined by \eqref{stea1}. Then, we deduce
$$
\int_{V} f_1(t,x,v) dv=0,\,\,\  \int_{V}vf_1(t,x,v) dv=0.
$$
Then, we assume the following asymptotic expansion in order 1 in $\varepsilon$,
\begin{equation}\label{hypo}
H(\phi+\varepsilon \psi,\theta)=H(\phi,\theta)+O(\varepsilon) \; \text{and} \; G_i(\phi+\varepsilon \psi,\theta)=G_i(\phi,\theta)+O(\varepsilon).
\end{equation}
Replacing now $f$ by its expansion $ f(t,x,v)=F_{n(t,x),u(t,x)}(v)+\varepsilon f_1(t,x,v)$ and using the first equality of \eqref{hypo}, yields
\begin{small}
\begin{eqnarray}\label{hyp22}
\begin{cases}
\displaystyle
\partial_t n+div_x(nu)=\int_{V}H(F_{n,u},g)dv+O(\varepsilon),\\
{}\\
\displaystyle
\partial_t(nu)+div_x \int_{V}v\otimes v F_{n,u}(v) dv= \sum_{i=1}^{m}\int_{V} v L_{i}^{1}[g_{i}](F_{n,u})dv
+ \int_{V}vH(F_{n,u},g)dv +O(\varepsilon).
\end{cases}
\end{eqnarray}\end{small}
Therefore
$$
\int_{V}v\otimes v F_{n,u}(v) dv=\int_{V} (v-u)\otimes (v-u) F_{n,u}(v) dv + nu\otimes u= P+ nu\otimes u,
$$
where the pressure tensor $P$ is given by
\begin{equation} \label{press1}
P(t,x)=\int_{V}(v-u(t,x))\otimes (v-u(t,x))F_{n(t,x),u(t,x)}(v) dv.
\end{equation}
Since $L_{i}^{1}$ conserves the local mass \eqref{mass}, the system \eqref{hyp22} becomes
\begin{small}
\begin{eqnarray}\label{hyp23}
\begin{cases}
\vspace{0.25cm}
\displaystyle
\partial_t n+div_x(nu)=\int_{V}H(F_{n,u},g)dv+O(\varepsilon),\\
\begin{array}{ll}
\displaystyle
\partial_t(nu)+div_x(P+nu\otimes u)=
& \displaystyle \sum_{i=1}^{m}\int_{V} (v-u)L_{i}^{1}[g_{i}](F_{n,u})dv  \\
& \displaystyle +\int_{V}vH(F_{n,u},g)dv +O(\varepsilon).
\end{array}
\end{cases}
\end{eqnarray}\end{small}
\begin{remark}\label{remark}
It is easy to see that the influence of the turning operator $ L^{0}$ on the macroscopic equations \eqref{hyp23} only comes into play
through the stationary state $F_{n,u}$ in the computation of the right-hand side of the second equation in \eqref{hyp23} and the pressure
tensor $P$. While the structure of the turning operator $L_{i}^{1}$ determines the effect of the chemical cues.
\end{remark}

Taking into account the system \eqref{hyp23} and using the second equality of \eqref{hypo} the system \eqref{systb} reads now
\begin{small}
\begin{eqnarray}\label{syfo}
\begin{cases}
\displaystyle
\partial_t n+div_x(nu)=\int_{V}H(F_{n,u},g)dv+O(\varepsilon),\\
{}\\
\displaystyle
\partial_t(nu)+div_x(P+nu\otimes u)=\sum_{i=1}^{m}\int_{V} (v-u)L_{i}^{1}[g_{i}](F_{n,u})dv+\int_{V}vH(F_{n,u},g)dv +O(\varepsilon),\\
{}\\
\displaystyle
\tau_{i} \partial_t N_{i} + div_x(N_{i}U_i)=\int_{V} G_i(F_{n,u},g)dv +O(\varepsilon),\\
{}\\
\displaystyle
\tau_{i} \partial_t(N_{i}U_i)+ div_x (Q(g_i))= \int_{V} v l_i(g_i)dv+\int_{V} v G_i(F_{n,u},g)dv+O(\varepsilon),\\
\end{cases}
\end{eqnarray}\end{small}
with
\begin{equation*}
Q(g_i):=\int_{V} v\otimes v g_{i}(t,x,v) dv=\big(\int_{V} v^k v^l g_{i}(t,x,v) dv\big )_{1\leq  k,l \leq d}.
\end{equation*}

It can be observed that system \eqref{syfo} is not yet closed. Indeed, it can be closed by looking for an approximate expression of $Q(g_i)$. The approach consists in deriving a function $a_i(t,x,v)$ which minimizers the $L^2(V)$-norm under the constraints that it has the same first moments, $N_i$ and $N_iU_i$, as $g_i$. Once $a_i$ this function has been found, we replace $Q(g_i)$ by $Q(a_i)$, and $g$ by $a$ in the others terms.

Toward this aim, we consider the set of velocities $V=s \mathbb{S}^{d-1}$ with $s>0$ and $\mathbb{S}^{d-1}$ the unit sphere of $\mathbb{R}^{d}$. Let us introduce Lagrangian multipliers $\eta_i$ and $\displaystyle{\overrightarrow{\xi_i}=(\xi_i^1,\cdots,\xi_i^d)}$ respectively scalar and vector, and define the following operator:
 \begin{eqnarray*}
M(a_i)=\frac{1}{2}&\int_{V} a_{i}^2(t,x,v)dv-\eta_i(\int_{V} a_{i}(t,x,v)dv-N_{i}(t,x,v) )\\
&-\overrightarrow{\xi_i}.(\int_{V} v a_{i}(t,x,v)dv-N_{i}(t,x,v)U_{i}(t,x,v) ).
\end{eqnarray*}

The Euler-Lagrange equation (first variation) of $M(a_i)$ reads $a_i=\eta_i+\overrightarrow{\xi_i}.v$. We use the constraints to define $\eta_i$ and $\overrightarrow{\xi_i}$. First, from the first equality in \eqref{hyd2} one gets easily $\eta _i=\frac{N_i}{|V|}$. Next, from Lemma \ref{lem} one obtains
$$
N_i(t,x)U_i(t,x)=\int_{V} v  a_{i}(t,x,v)dv=|V|\frac{s^2}{d}\overrightarrow{\xi_i},
$$
then $\overrightarrow{\xi_i}=\frac{d}{|V|s^2}N_i(t,x)U_i(t,x)$.

 Therefore,
\begin{equation}\label{clo1}
a_i(t,x,v)=\frac{1}{|V|}\left(N_i(t,x)+\frac{d}{s^2} N_i(t,x)U_i(t,x).v \right).
\end{equation}
Consequently, using again lemma \ref{lem}, the pressure tensor $ Q(a_i)$  is
 $$Q(a_i)=\int_{V} v\otimes v a_i(t,x,v) dv=\frac{1}{|V|}\int_{V} v\otimes v N_{i}dv=\frac{s^2}{d}N_{_i}\mathbb{I}_d,$$
where $\mathbb{I}_d$ denotes the $d\times d$ identity matrix.
Thus, the following nonlinear coupled hyperbolic model is derived:
\begin{small}
\begin{eqnarray}\label{nwsyst}
\begin{cases}
\displaystyle
\partial_t n+div_x(nu)=\int_{V}H(F_{n,u},a)dv+O(\varepsilon),\\
{}\\
\displaystyle
\partial_t(nu)+div_x(P+nu\otimes u)=\sum_{i=1}^{m}\int_{V} (v-u)L_{i}^{1}[a_{i}](F_{n,u})dv+ \int_{V}vH(F_{n,u},a)dv+O(\varepsilon),\\
{}\\
\displaystyle
\tau_{i} \partial_t N_{i} + div_x(N_{i}U_i)=\int_{V} G_i(F_{n,u},a)dv+O(\varepsilon),\\
{}\\
\displaystyle
\tau_{i} \partial_t(N_{i}U_i)+\frac{s^2}{d}\nabla_{x} N_{_i} = \int_{V} v l_i(a_i)dv+\int_{V} v G_i(F_{n,u},a)dv+O(\varepsilon),\\
\end{cases}
\end{eqnarray}\end{small}
with $a=(a_1,\cdots,a_m)$.
\begin{remark}
The second variation of $M$ is $\delta^2 M(a_i)=1$, then the extremum $a_i(t,x,v)$ is a minimum.
\end{remark}
%

\section{Derivation of models}
This section shows how the tools reviewed in the preceding section can be used to derive models. Let us consider the model defined by choosing the stationary state and the turning kernels. Consider $F_{n,u}$ as follows:
 \begin{equation}\label{clo2}
F_{n,u}(v)=\frac{1}{|V|}(n+\frac{d}{s^2}nu.v),
\end{equation}
It is easy to check that $F_{n,u}$ satisfies the assumptions \eqref{flu1}-\eqref{stea1}.\\
We take the turning kernel $T^{0}$ in \eqref{n0} in the form
$$
T^{0}(v,v')=\frac{\mu_0}{|V|}(1+\frac{d}{s^2}v \cdot v'),
$$
with $\mu_0$ a real constant, and consider that the turning kernel $ T_{i}^{1}$ in \eqref{n1} depends on the velocity $v'$, on the population $g_i$, and on its gradient, defined by:
$$
T_{i}^{1}[g_i](v,v')=\frac{\mu_1}{|V|}-\frac{\mu_{2}d}{|V| s^2}v'\cdot \alpha(<g_i>),
$$
where $\alpha $ is a mapping $\mathbb{R} \longrightarrow \mathbb{R}^d$, $\mu_1,\,\mu_2 $ are real constants and $<\cdot>$ stands for the $(v)$-mean of
a function, i.e $\displaystyle{<h>:=\int_{V} h(t,x,v)dv}$ for $h \in L^2(V)$.\\
Therefore, the turning operator $L^0$ is given by:
\begin{align}
L^{0}(f)&=\int_{V}\left(T^{0}(v,v')f(t,x,v')-T^{0}(v',v)f(t,x,v)\right) dv  \nonumber\\
&=\mu_0 \left( \frac{1}{|V|}(n+\frac{d}{s^2}nv \cdot u)-f(v)\right)\nonumber\\
&=\mu_0 \left(F_{n,u}(v)-f(v)\right),
\end{align}
then, $L^{0}$ is a relaxation operator to $F_{n,u}$.\\
While, the turning operator $L_{i}^{1}[g_{i}]$ can be computed as follows:
\begin{align*}
L_{i}^{1}[g_{i}](f)&=\int_{V}\left(T_{i}^{1}(g_{i},v,v')f(t,x,v')-T_{i}^{1}(g_{i},v',v)f(t,x,v)\right)dv'\\
&=\frac{\mu_1}{|V|}n-\frac{\mu_{2}d}{|V| s^2}nu.\alpha(<g_i>)-\mu_1f(v)+\frac{\mu_{2}d}{ s^2}vf(v)\cdot \alpha(<g_i>)\\
&=\mu_1\left(\frac{n}{|V|}-f(v)\right)-\frac{\mu_{2}d}{ s^2}\left(\frac{nu}{|V|}-vf(v)\right)\cdot \alpha(<g_i>).
\end{align*}
Thus,
\begin{align*}
\quad \quad \int_{V}&(v-u) L_{i}^{1}[g_{i}](F_{n,u})dv\\
&= \mu_1 \int_{V} v \left(\frac{n}{|V|}-F_{n,u}(v)\right)dv-\frac{\mu_{2}d}{ s^2} \int_{V} v\left(\frac{nu}{|V|}-vF_{n,u}(v)\right)\cdot \alpha(<g_i>)dv\\
&=-\mu_1 nu+ \mu_2 n \alpha(<g_i>).
\end{align*}
Consequently,
\begin{equation}\label{tu1}
\sum_{i=1}^{m} \int_{V} (v-u)L_{i}^{1}[g_{i}](F_{n,u})dv= -\mu_1 m nu+ \sum_{i=1}^{m}\mu_2 n \alpha(<g_i>).
\end{equation}
Finally, take the turning kernel $K_i$ in \eqref{n2} as follows:
$$K_i(v,v')=\frac{\sigma_i}{|V|},\,\,\, $$
with $\sigma_i$ is a real constant.\\
Then, the turning operator $l_i$ is computed as follows:
\begin{equation*}
l_i(h)=\int_{V} \left(K_i(v,v')h(t,x,v')-K_i(v',v)h(t,x,v)\right)dv'=\sigma_i \left(\frac{<h>}{|V|}- h \right).
\end{equation*}
Therefore,
\begin{align*}
\int_{V} v l_i(h)dv
&=- \sigma_i \int_{V} v h dv.
\end{align*}
Consequently,
\begin{equation}
\int_{V} v l_i(a_i)dv=-\sigma_i N_iU_i.
\end{equation}
Now we compute the pressure tensor $P$. By using lemma \ref{lem}, we have
\begin{equation*}
\int_{V} v\otimes v F_{n,u} dv=\int_{V} v\otimes v\frac{1}{|V|}(n+\frac{d}{s^2} nu.v)=\frac{s^2}{d}n \mathbb{I}_d.
\end{equation*}
Thus,
\begin{equation}\label{press11}
P+nu\otimes u=\frac{s^2}{d}n \mathbb{I}_d.
\end{equation}
Finally, the system \eqref{nwsyst} becomes, at first order with respect to $\varepsilon$,
\begin{equation}\label{crsys}
\begin{cases}
\vspace{0.25cm}
\displaystyle
\partial_t n+div_x(nu)=\int_{V}H(F_{n,u},a) dv,\\
\vspace{0.25cm}
\displaystyle
\partial_t(nu)+\frac{s^2}{d}\nabla_{x} n=-\mu_1 m nu+ \mu_2 \sum_{i=1}^{m} n \alpha(N_i)+\int_{V}v H(F_{n,u},a) dv,\\
\vspace{0.25cm}
\displaystyle
\tau_{i} \partial_t N_{i}+ div_x(N_{i}U_i)=\int_{V} G_i(F_{n,u},a)dv,\\
\displaystyle
\tau_{i} \partial_t(N_{i}U_i)+\frac{s^2}{d}\nabla_{x} N_{_i} = -\sigma N_iU_i+\int_{V}v G_i(F_{n,u},a)dv,
\end{cases}
\end{equation}
where $a$ and $F_{n,u}$ are defined in \eqref{clo1} and \eqref{clo2}.
\begin{theorem}
 If we consider for all $i=1,\cdots,m$, $\alpha(N_i)=\alpha_i(N_i)\nabla_{x} N_{_i}$, $H$ and  $G_i$ satisfy the assumption \eqref{hypo} then, we obtain
 the following system at first order with respect to $\varepsilon$,
\begin{eqnarray}\label{cr1sys}
\begin{cases}
\vspace{0.25cm}
\displaystyle
\partial_t n+div_x(nu)=\int_{V}H(F_{n,u},a) dv,\\
\vspace{0.25cm}
\displaystyle
\partial_t(nu)+\frac{s^2}{d}\nabla_{x} n=-\mu_1 m nu+ \mu_2 \sum_{i=1}^{m} n \alpha_i(N_i)\nabla_{x} N_{_i}+\int_{V}v H(F_{n,u},a) dv,\\
\vspace{0.25cm}
\displaystyle
\tau_{i} \partial_t N_{i} + div_x(N_{i}U_i)=\int_{V} G_i(F_{n,u},a)dv,\\
\displaystyle
\tau_{i} \partial_t(N_{i}U_i)+\frac{s^2}{d}\nabla_{x} N_{i} = -\sigma_i N_iU_i+\int_{V}v G_i(F_{n,u},a)dv.
\end{cases}
\end{eqnarray}
\end{theorem}

This theorem leads to some specific models which are presented in the next subsection.
\subsection{A Cattaneo type model for chemosensitive movement}
Taking $m=1$ in \eqref{cr1sys}, one can derive the corresponding hyperbolic
system for chemosensitive movement, at first order with respect to $\varepsilon$, as follows 
\begin{equation}\label{syst1.}
\begin{cases}
\vspace{0.25cm}
\displaystyle
\partial_t n+div_x(nu)=\Psi(F_{n,u},a_1),\\
\vspace{0.25cm}
\displaystyle
\partial_t(nu)+\frac{s^2}{d}\nabla_{x} n=-\mu_1 nu+ \mu_2  n \alpha_1(N_1)\nabla_{x} N_1+\widetilde{\Psi}(F_{n,u},a_1),\\
\vspace{0.25cm}
\displaystyle
\tau_{1} \partial_t N_1+ div_x(N_1U_1)=\Phi_{1}(F_{n,u},a_1),\\
\displaystyle
\tau_{1} \partial_t(N_1U_1)+\frac{s^2}{d}\nabla_{x} N_1 = -\sigma_1 N_1U_1+\widetilde{\Phi}_{1}(F_{n,u},a_1),
\end{cases}
\end{equation}
where $$\Psi(F_{n,u},a_1):=\int_{V}H(F_{n,u},a_1)dv,\,\  \widetilde{\Psi}(F_{n,u},a_1):=\int_{V}v H(F_{n,u},a_1)dv$$ and
$$\Phi_{1}(F_{n,u},a_1):=\int_{V} G_1(F_{n,u},a_1)dv,\,\  \widetilde{\Phi}_{1}(F_{n,u},a_1):=\int_{V} v G_1(F_{n,u},a_1)dv.$$

In absence of interactions, the authors in \cite{filbert1} and \cite{hillen4} derived, respectively, the first two equations
for $(n,nu)$ by asymptotic analysis and moment closure. The system composed by the first two equations with $H=0$ is called the Cattaneo model for chemosensitive movement with density control \cite{dolak,hillen2}.

\subsection{Derivation of Keller-Segel models} \label{sec 4.2}
The approach proposed can be applied to derive a variety of models of Keller-Segel type. Indeed, by taking the system \eqref{cr1sys} and with different scalings this approach allows to derive various models.\\
From \eqref{clo1} and \eqref{clo2}, we have $\displaystyle{F_{n,u}=\frac{1}{|V|}(n+\frac{d}{s^2}nu\cdot v)}$ and $a_1=\frac{1}{|V|}(N_1+\frac{d}{s^2}N_1U_1\textcolor{blue}{\cdot}v)$. To get our aim, we assume moreover in this subsection the following assumption,
\begin{equation}\label{assu1}
H(F_{n,u},a_1)=H\left(\frac{n}{|V|},\frac{N_1}{|V|}\right)+O(\frac{1}{s^2})\,\,\text{and}\,\,\ G_1(F_{n,u},a_1)=G_1\left(\frac{n}{|V|},\frac{N_1}{|V|}\right)+O(\frac{1}{s^2}),
\end{equation}
and we set
\begin{equation*}
\widetilde{H}(n,N_1)=|V|H\left(\frac{n}{|V|},\frac{N_1}{|V|}\right) \quad \text{and} \quad \widetilde{G_1}(n,N_1)=|V|G_1\left(\frac{n}{|V|},\frac{N_1}{|V|}\right).
\end{equation*}
Consequently, we have the following proposition,
\begin{proposition}
For $m=1$, the system \eqref{cr1sys} becomes, with above assumptions \eqref{hypo} and \eqref{assu1}, which are satisfied if $H$ and $G_1$ are bilinear
\begin{equation}\label{systm11}
\begin{cases}
\vspace{0.25cm}
\displaystyle
\partial_t n+div_x(nu)= \widetilde{H}(n,N_1)+O(\frac{1}{s^2}),\\
\vspace{0.25cm}
\displaystyle
\partial_t(nu)+\frac{s^2}{d}\nabla_{x} n=-\mu_1  nu+ \mu_{2} n \alpha_1(N_1)\nabla_{x} N_{_1}+ O(\frac{1}{s^2}),\\
\vspace{0.25cm}
\displaystyle
\tau_{1} \partial_t N_{1} + div_x(N_{1}U_1)=  \widetilde{G_1}(n,N_1)+O(\frac{1}{s^2}),\\
\displaystyle
\tau_{1} \partial_t (N_{1}U_1)+\frac{s^2}{d}\nabla_{x} N_{_1} = -\sigma_1 N_1U_1+O(\frac{1}{s^2}).
\end{cases}
\end{equation}
\end{proposition}

Let now $ \sigma_1 \rightarrow\infty$ and $ s \rightarrow\infty $ such that $\frac{s^2}{d \sigma_1}\rightarrow D_{N_{1}}$. Dividing the fourth equation of system \eqref{systm11} by $\sigma_1$ and taking last limits, yields
$
D_{N_{1}}\nabla_{x} N_{_1}=-N_1U_1,
$
therefore the third equation of \eqref{systm11} writes
\begin{equation}\label{sub1}
\tau_{1} \frac{\partial N_{1} }{\partial t}-D_{N_{1}}\Delta_{x} N_{_1}= \widetilde{G_1}(n,N_1).
\end{equation}
Thus, we get the following system
\begin{equation}\label{sysm11}
\begin{cases}
\vspace{0.25cm}
\displaystyle
\partial_t n+div_x(nu)=\widetilde{H}(n,N_1)+O(\frac{1}{s^2}),\\
\vspace{0.25cm}
\displaystyle
\partial_t(nu)+\frac{s^2}{d}\nabla_{x} n=-\mu_1  nu+ \mu_{2} n \alpha_1(N_1)\nabla_{x} N_{_1}+ O(\frac{1}{s^2}),\\
\displaystyle
\tau_{1} \partial_t N_{1} -D_{N_{1}}\Delta_{x} N_{_1}= \widetilde{G_1}(n,N_1) + O(\frac{1}{s^2}).
\end{cases}
\end{equation}

Consequently, if we take: $\tau_1=1,\,\ \alpha_1(N_1)=1$ and $H=O(\frac{1}{s^2}) =0$  and we define
$$\widetilde{G_1}(n,N_1)=g(n,N_1),$$
then we recover the system $(16)$ in \cite{filbert1}.\\

In addition we apply an other scaling for the two first equations of \eqref{sysm11} we can derive some K-S type models. Indeed, we take $\mu1=\mu2$ and $s \rightarrow \infty $ such that
\begin{equation}\label{limit}
 \frac{s^2}{d \mu_1}\rightarrow D_{n}.
\end{equation}

Next, dividing the second equation in \eqref{sysm11} by $\mu_1$ and taking last limits, yields
$$
D_{n}\nabla_{x} n= - nu+\alpha_1(N_1) n \nabla_{x} N_1,
$$
then,
$$
nu =\alpha_1(N_1) n \nabla_{x} N_1- D_{n}\nabla_{x} n,
$$
replacing in the first equation  of system \eqref{sysm11}, with $S:=N_1, \chi(S):=\alpha_1(S)$, yields
\begin{equation}\label{sysm111}
\begin{cases}
\vspace{0.25cm}
\displaystyle
\partial_t n=div_x (D_{n}\nabla_{x} n- n \chi(S)\nabla_{x} S)+\widetilde{H}(n,S),\\
\displaystyle
\tau_{1} \partial_t S=D_{S}\Delta_{x} S+\widetilde{G_1}(n,S).
\end{cases}
\end{equation}
System \eqref{sysm111} consists of two coupled reaction-diffusion equations, which are parabolic equations. Moreover, this model is one of the simplest models to describe the aggregation of cells by chemotaxis.

\section{Numerical methods}
Now, we present some numerical tests in the hyperbolic model \eqref{cr1sys} with the choice $m=1$, $H=0$, $G_1=\frac{n}{|V|}$, and $\tau_1=1$:
\begin{equation}\label{mod.n}
\begin{cases}
\vspace{0.25cm}
\displaystyle
\partial_tn+div_x(nu)=0,\\
\vspace{0.25cm}
\displaystyle
\partial_t(nu)+\frac{s^2}{d}\nabla_{x} n=-\mu_1 nu+ \mu_2  n \alpha_1(N_1)\nabla_{x} N_1,\\
\vspace{0.25cm}
\displaystyle
\partial_tN_1 + div_x(N_1U_1)=n,\\
\displaystyle
\partial_t(N_1U_1)+\frac{s^2}{d}\nabla_{x} N_1 = -\sigma_1 N_1U_1.\\
\end{cases}
\end{equation}
To compute numerical solutions of \eqref{mod.n} in one space dimension we use a well-balanced scheme adapting the method developed by Gosse and Toscani \cite{gosse-toscani}. Well-balanced schemes have been developed in order to guarantee good behaviour of numerical solutions for large time \cite{Laurent-Gosse}. Moreover, we show that the resulting scheme is asymptotic preserving for the limit in \eqref{limit}, in the sense that it is asymptotically equivalent to a well-balanced numerical scheme for the Keller-Segel model. The two-dimensional case referring to \cite{filbert1}, where the numerical method is based on time splitting scheme between the conservative part and the source term of system \eqref{mod.n} where the conservative equation is approximated by the Lax-Friedrichs scheme \cite{Elena-Vazquez,LeVeque}.

\subsection{One dimensional well-balanced and asymptotic-preserving scheme}
In this section we present a well-balanced discretization of system \eqref{mod.n} in one-dimensional setting subject to the scaling of Section \ref{sec 4.2}.
The scheme obtained is asymptotic preserving in the sense that when \eqref{limit}
holds, the limiting scheme is asymptotically equivalent to the well-known Scharfetter-Gummel scheme for the Keller-Segel equations \eqref{sysm111}.

Let us first give an other presentation of system \eqref{mod.n}. We are in the setting of  Section \ref{sec 4.2}, so we set
\begin{equation}\label{n.5.1}
\mu_1=\mu_2=\frac{s^2}{D_n}, \quad \text{and} \quad \sigma_1=\frac{s^2}{D_{N_1}}.
\end{equation}
System \eqref{mod.n} in one dimension, replacing $\mu_1$, $\mu_2$ and $\sigma_1$ by their expressions in \eqref{n.5.1}, yields
\begin{equation}\label{n.5.2}
\begin{cases}
\vspace{0.25cm}
\displaystyle
\partial_t n+ \partial_x(nu)=0,\\
\vspace{0.25cm}
\displaystyle
\varepsilon^2 \partial_t(nu)+ \partial_x n=an-\frac{nu}{D_n},\\
\vspace{0.25cm}
\displaystyle
\partial_t N_{1} + \partial_x(N_{1}U_1)=n,\\
\displaystyle
\varepsilon^2 \partial_t(N_{1}U_1)+\partial_xN_{1}= -\frac{N_1U_1}{D_{N_1}},
\end{cases}
\end{equation}
\vspace{0.25cm}
with $\varepsilon = \frac{1}{s}$ and  $a=\frac{\alpha_1}{D_n}\partial_xN_{1}$.\\
Following the ideas of \cite{gosse-toscani}, we write \eqref{n.5.2} as
\begin{equation}\label{n.5.3}
\begin{cases}
\vspace{0.25cm}
\displaystyle
\partial_tv+\frac{1}{\varepsilon} \partial_x w=\frac{1}{2\varepsilon}\big[ (a-\frac{1}{\varepsilon D_n})v + (a+\frac{1}{\varepsilon D_n})w \big],\\
\vspace{0.25cm}
\displaystyle
\partial_t w - \frac{1}{\varepsilon}\partial_xv=-\frac{1}{2\varepsilon}\big[ (a-\frac{1}{\varepsilon D_n})v + (a+\frac{1}{\varepsilon D_n})w \big],\\
\vspace{0.25cm}
\displaystyle
\partial_tV+\frac{1}{\varepsilon} \partial_x W=-\frac{1}{2\varepsilon^2D_{N_1}}(V-W)+\frac{n}{2},\\
\displaystyle
\partial_tW-\frac{1}{\varepsilon} \partial_xV= \frac{1}{2\varepsilon^2D_{N_1}}(V-W)+ \frac{n}{2},
\end{cases}
\end{equation}
where
\begin{equation}\label{n.5.4}
v=\frac{1}{2}(n+\varepsilon (nu)), \quad V=\frac{1}{2}(N_1+\varepsilon (N_1U_1)),
\end{equation}
\begin{equation}\label{n.5.5}
w=\frac{1}{2}(n-\varepsilon (nu)), \quad  W=\frac{1}{2}(N_1-\varepsilon (N_1U_1)).
\vspace{0.25cm}
\end{equation}

We are now ready to deduce a numerical discretization of system \eqref{mod.n} based in the representation \eqref{n.5.3}.
We discretize $[0,T]\times [-L,L]$, $T,L>0$, by a uniform Cartesian computational grid determined by $\Delta x$ and $\Delta t$, standing for the space and time steps respectively. Let $x_i$ and $t^k$ such that $x_i=-L+i\Delta x$ and $t^k=k\Delta t$, $i=0,\cdots,N_x$, $k\in \mathbb{N}$. The approximations of $v(x,t)$, $w(x,t)$, $V(x,t)$ and $W(x,t)$ at the spatial point $x_i$ and at the time step $t^k$ are denoted by $v_i^k \approx v(t_k,x_i)$, $w_i^k \approx w(t_k,x_i)$, $V_i^k \approx V(t_k,x_i)$ and $W_i^k \approx W(t_k,x_i)$ respectively. We will recover approximations of $n(x,t)$, $nu(x,t)$, $N_1(x,t)$ and $N_1 U_1(x,t)$ by setting $\displaystyle{n_i^k = v_i^k+w_i^k}$, $\displaystyle{(nu)_i^k=\frac 1\varepsilon (v_i^k-w_i^k)}$, \vspace{0.25cm} $\displaystyle{N_1{}_i^k = V_i^k+W_i^k}$, $\displaystyle{(N_1U_1){}_i^k = \frac 1\varepsilon (V_i^k-W_i^k)}$.

Following the ideas in Gosse-Toscani \cite{gosse-toscani}, we discretize \eqref{n.5.3} by
\begin{equation}\label{n.5.6}
\begin{cases}
\displaystyle
\vspace{0.2cm}
v^{k+1}_i=v^k_i - \frac{\Delta t}{\varepsilon \Delta x}(v^{k+1}_i-v^{k+\frac{1}{2}}_{i-\frac{1}{2}}),\\
\displaystyle
\vspace{0.2cm}
w^{k+1}_{i-1}=w^k_{i-1} - \frac{\Delta t}{\varepsilon \Delta x}(w^{k+1}_{i-1}-w^{k+\frac{1}{2}}_{i-\frac{1}{2}}),\\
\displaystyle
\vspace{0.2cm}
V^{k+1}_i=V^k_i - \frac{\Delta t}{\varepsilon \Delta x}(V^{k+1}_i-V^{k+\frac{1}{2}}_{i-\frac{1}{2}}) + \frac{\Delta t}{2} n^{k+1}_i,\\
\displaystyle
W^{k+1}_{i-1}=W^k_{i-1} - \frac{\Delta t}{\varepsilon \Delta x}(W^{k+1}_{i-1}-W^{k+\frac{1}{2}}_{i-\frac{1}{2}})+\frac{\Delta t}{2} n^{k+1}_i,
\end{cases}
\end{equation}
with $i=0,\cdots,N_x$. In order to update the values $v^k_i$, $w^k_{i-1}$, $V^k_i$, $W^k_{i-1}$, we need expressions for the numerical flux $v_{i-\frac{1}{2}}$, $w_{i-\frac{1}{2}}$, $V_{i-\frac{1}{2}}$ and $W_{i-\frac{1}{2}}$. For that purpose we solve in $[x_{i-1},x_i]$, the stationary problem composed of the four equations of \eqref{n.5.3}
\begin{equation*}\label{n.5.7}
\begin{cases}
\displaystyle
\vspace{0.2cm}
\partial_x \bar{v}=\frac{1}{2}\Big[ (a_{i-\frac{1}{2}}-\frac{1}{\varepsilon D_n})\bar{v}  + (a_{i-\frac{1}{2}}+\frac{1}{\varepsilon D_n})\bar{w} \Big],\\
\displaystyle
\vspace{0.2cm}
\partial_x\bar{w}=\frac{1}{2}\Big[ (a_{i-\frac{1}{2}}-\frac{1}{\varepsilon D_n})\bar{v}  + (a_{i-\frac{1}{2}}+\frac{1}{\varepsilon D_n})\bar{w} \Big],\\
\displaystyle
\vspace{0.2cm}
\partial_x \bar{V}=-\frac{1}{2\varepsilon D_{N_1}}(\bar{V}-\bar{W}),\\
\displaystyle
\partial_x\bar{W}=-\frac{1}{2\varepsilon D_{N_1}}(\bar{V}-\bar{W}),\\
\end{cases}
\end{equation*}
where, \vspace{0.25cm}$a_{i-\frac{1}{2}}=\frac{\alpha_1}{D_n}\frac{N_{1,i}-N_{1,i-1}}{\Delta x}$, $i=0,\cdots,N_x$. \\
We complete this system with the incoming boundary conditions
\begin{equation*}
\bar{v}(x_{i-1})=v_{i-1}, \;\; \bar{V}(x_{i-1})=V_{i-1}, \;\;
\bar{w}(x_i)=w_i, \;\;  \bar{W}(x_i)=W_i,
\end{equation*}
and we look for the unknowns:
\begin{equation*}
v_{i-\frac{1}{2}}=\bar{v}(x_i), \;\; V_{i-\frac{1}{2}}=\bar{V}(x_i), \;\; w_{i-\frac{1}{2}}=\bar{w}(x_{i-1}),\;\; W_{i-\frac{1}{2}}=\bar{W}(x_{i-1}).
\end{equation*}
One can solve explicitely this system of differential equations.
After straightforward but tedious computations, one finds
\begin{equation}\label{n.5.8.1}
v_{i-\frac{1}{2}}=\;w_i + f_{i-\frac{1}{2}}, \quad \;\;\;V_{i-\frac{1}{2}}=\;W_i + F_{i-\frac{1}{2}}, \quad i=0,\cdots,N_x
\end{equation}
\begin{equation}\label{n.5.8.2}
w_{i-\frac{1}{2}}=\;v_{i-1} - f_{i-\frac{1}{2}}, \quad W_{i-\frac{1}{2}}=\;W_{i-1} - F_{i-\frac{1}{2}}, \quad i=0,\cdots,N_x,
\end{equation}
where
\begin{equation*}
f_{i-\frac{1}{2}} = \frac{2\varepsilon a_{i-\frac{1}{2}}D_n \big(v_{i-1}-e^{-a_{i-\frac{1}{2}}\Delta x}w_{i}\big)}{\varepsilon a_{i-\frac{1}{2}}(1+e^{-a_{i-\frac{1}{2}}\Delta x}) - (e^{-a_{i-\frac{1}{2}}\Delta x}-1)},
\end{equation*}
\begin{equation*}
\text{and} \quad \quad F_{i-\frac{1}{2}} = \frac{2\varepsilon D_{N_1}}{2\varepsilon D_{N_1} + \Delta x}(V_{i-1}-W_{i}). \hspace{1.5cm}
\vspace{0.25cm}
\end{equation*}
Now the approximations of the numerical fluxes
$v_{i-\frac{1}{2}}^{k+\frac{1}{2}}$, $w_{i-\frac{1}{2}}^{k+\frac{1}{2}}$, $V_{i-\frac{1}{2}}^{k+\frac{1}{2}}$ and $W_{i-\frac{1}{2}}^{k+\frac{1}{2}}$ are computed from \eqref{n.5.8.1}, \eqref{n.5.8.2} as
\begin{equation}\label{n.5.11.1}
v^{k+\frac{1}{2}}_{i-\frac{1}{2}}=\;w_i^{k+1} + f_{i-\frac{1}{2}}^k, \quad V_{i-\frac{1}{2}}^{k+\frac{1}{2}}=\;W_i^{k+1} + F_{i-\frac{1}{2}}^{k+1}, \quad i=0,\cdots,N_x
\end{equation}
\begin{equation}\label{n.5.11.2}
w_{i-\frac{1}{2}}^{k+\frac{1}{2}}=\;v_{i-1}^{k+1} - f_{i-\frac{1}{2}}^k, \quad W_{i-\frac{1}{2}}^{k+\frac{1}{2}}=\;V_{i-1}^{k+1} - F_{i-\frac{1}{2}}^{k+1}, \quad i=0,\cdots,N_x,
\end{equation}
with
\begin{equation}\label{n.5.12.1}
f_{i-\frac{1}{2}}^k = \frac{2\varepsilon a_{i-\frac{1}{2}}^k D_n(v_{i-1}^k-e^{-a_{i-\frac{1}{2}}^k\Delta x}w_i^k)}{\varepsilon a_{i-\frac{1}{2}}^k(1+e^{-a_{i-\frac{1}{2}}^k\Delta x}) - (e^{-a_{i-\frac{1}{2}}^k \Delta x}-1)},
\end{equation}
\begin{equation}\label{n.5.12.2}
F_{i-\frac{1}{2}}^{k+1} = \frac{2\varepsilon D_{N_1}}{2\varepsilon D_{N_1} + \Delta x}(V_{i-1}^{k+1}-W_i^{k+1}), \quad \text{and} \quad
a^k_{i-\frac{1}{2}}=\frac{\alpha_1}{D_n} \frac{N_{1,i}^k-N_{1,i-1}^k}{\Delta x}.
\end{equation}
Since in \eqref{n.5.6} the numerical fluxes are multiplied by a factor of
order $\frac 1\varepsilon$, we use in \eqref{n.5.11.1}-\eqref{n.5.11.2} a semi-implicit
discretization in time where the term $f_{i-\frac 12}$, which is of order
$\varepsilon$, is treated explicitly.
From \eqref{n.5.6}, \eqref{n.5.8.1} and \eqref{n.5.8.2} we obtain, for $i=0,\cdots,N_x$, the following well-balanced scheme of system \eqref{n.5.3}
\begin{equation}\label{n.5.14}
\begin{cases}
\vspace{0.2cm}
\displaystyle
(1+\frac{\Delta t}{\varepsilon \Delta x}) v_i^{k+1} - \frac{\Delta t}{\varepsilon \Delta x} w_i^{k+1}= v_i^k + \frac{\Delta t}{\varepsilon \Delta x} f^k_{i-\frac{1}{2}},\\
\vspace{0.2cm}
\displaystyle
(1+\frac{\Delta t}{\varepsilon \Delta x}) w_i^{k+1} - \frac{\Delta t}{\varepsilon \Delta x} v_i^{k+1}= w_i^k - \frac{\Delta t}{\varepsilon \Delta x} f^k_{i+\frac{1}{2}},\\
\vspace{0.2cm}
\displaystyle
(1+\frac{\Delta t}{\varepsilon \Delta x}) V_i^{k+1} - \frac{\Delta t}{\varepsilon \Delta x} W_i^{k+1}= V_i^k + \frac{\Delta t}{\varepsilon \Delta x} F^{k+1}_{i-\frac{1}{2}} + \frac{\Delta t}{2} n^{k+1}_i,\\
\displaystyle
(1+\frac{\Delta t}{\varepsilon \Delta x}) W_i^{k+1} - \frac{\Delta t}{\varepsilon \Delta x} V_i^{k+1}= W_i^k - \frac{\Delta t}{\varepsilon \Delta x} F^{k+1}_{i+\frac{1}{2}}+\frac{\Delta t}{2} n^{k+1}_i,\\
\end{cases}
\end{equation}
\vspace{0.25cm}
where $f_{i-\frac{1}{2}}^k$ and $F_{i-\frac{1}{2}}^{k+1}$ are given in \eqref{n.5.12.1} and \eqref{n.5.12.2} respectively.\\
The ghost-points, points with index $i=-1$ or $i=N_x+1$, are computed from the boundary conditions where we impose Neumann boundary conditions for the density $n$ and for the concentration $N_1$
\begin{equation}\label{n.6.12}
\frac{\partial n}{\partial \eta} \Big|_{\partial \Omega}=0, \quad \text{and} \quad \frac{\partial N_1}{\partial \eta} \Big|_{\partial \Omega}=0,
\end{equation}
$\eta(x)$ stand for the inward unit normal at $x \in \partial \Omega$. The boundary conditions for the flux $q:=nu$ are the Dirichlet conditions:
\begin{equation}\label{n.6.13}
q|_{\partial \Omega} = 0.
\end{equation}
This yields
\begin{equation}\label{j.1}
n_{-1}^k=n_1^k, \quad n_{N_x+1}^k=n_{N_x-1}^k,
\end{equation}
\begin{equation}
q_{-1}^k=q_1^k, \quad q_{N_x+1}^k=q_{N_x-1}^k.
\end{equation}
Next, we will prove that \eqref{n.5.14} is asymptotic preserving scheme, more precisely we will prove that when $\varepsilon$ is small (i.e $s$ is large) \eqref{n.5.14} is asymptotically equivalent to the Scharfetter-Gummel scheme, discussed in bellow, for the Keller-Segel model.

We first recall Scharfetter-Gummel method \cite{ScharfGum} adapted to the Keller-Segel model.
It has been shown in section \ref{sec 4.2} that problem \eqref{cr1sys} is "asymptotically" equivalent to the following Keller-Segel type model
\begin{equation}\label{n.5.15}
\begin{cases}
\vspace{0.25cm}
\displaystyle
\partial_tn=\partial_x (D_{n}\partial_{x} n- n \chi(S)\partial_{x} S),\\
\displaystyle
\partial_t S=D_{S}\partial_{xx} S,
\end{cases}
\end{equation}
\vspace{0.25cm}
with $S=N_1$ and $\chi(S)=\alpha_1(S)$.\\
We rewrite the first equation of \eqref{n.5.15} as
\begin{equation}\label{n.5.16}
\partial_t n + \partial_x J=0, \quad \text{with} \quad J=-D_{n}\partial_{x} n + n \alpha_1\partial_{x} S.
\end{equation}
In standard notation, the discretization of the equation \eqref{n.5.15} writes
\begin{equation}\label{n.5.17}
\frac{n_i^{k+1} - n_i^k}{\Delta t} + \frac{J_{i+\frac{1}{2}}^k - J_{i-\frac{1}{2}}^k}{\Delta x} =0,
\end{equation}
Here, the flux $J^k_{i+\frac{1}{2}}$ is given by the local boundary-value problem
\begin{equation}\label{n.5.18}
\begin{cases}
\vspace{0.25cm}
\displaystyle
J^k_{i-\frac{1}{2}} = -D_n \partial_x\bar{n} + \alpha_1 \frac{S^k_{i+1}-S^k_i}{\Delta x} \bar{n},\\
\displaystyle
\bar{n}(0)=n^k_i, \; \bar{n}(\Delta x)=n^k_{i+1}.
\end{cases}
\vspace{0.25cm}
\end{equation}
This differential system can be solved explicitely, one gets
\begin{equation}\label{n.5.19}
J^k_{i+\frac{1}{2}} = \alpha_1 \partial_x^{(c)} S^k_i \frac{n^k_i - \exp(-\frac{\alpha_1 \Delta x}{D_n} \partial_x^{(c)} S^k_i)n^k_{i+1}}{1-\exp(-\frac{\alpha_1 \Delta x}{D_n} \partial_x^{(c)} S^k_i)},
\end{equation}
\vspace{0.25cm}
where, $\partial_x^{(c)} S^k_i=\frac{S^k_{i+1} - S^k_i}{\Delta x}$, $i=0,\cdots,N_x$.\\
On the other hand the second equation of system \eqref{n.5.15} is approximated by the classical second order finite difference scheme \cite{LeVeque}
\begin{equation}\label{n.5.20}
\frac{S^{k+1}_i - S^k_i}{\Delta t}= D_S \frac{S^{k+1}_{i-1} - 2S^{k+1}_i + S^{k+1}_{i+1}}{(\Delta x)^2} + n^{k+1}_i,
\vspace{0.25cm}
\end{equation}
\vspace{0.25cm}
with $i=0,\cdots,N_x$. On the boundaries, we again use \eqref{j.1}.\\
The next proposition show that the well-balanced scheme \eqref{n.5.14} is asymptotic preserving scheme.
\begin{proposition}\label{psp}
Formally, when $\varepsilon\to 0$, the numerical discretization \eqref{n.5.14} converges to the discretization \eqref{n.5.17}, \eqref{n.5.19}, \eqref{n.5.20} of Keller-Segel system \eqref{n.5.15} (with $S=N_1$ and $\chi=\alpha_1$).
\end{proposition}
\begin{proof}
By summing the first and second equations of \eqref{n.5.14} and the third and fourth equations, one can drive, for every $i=0,\cdots,N_x$, the two following equations
\begin{equation}\label{n.5.21.1}
n_i^{k+1} =n^k_i + \frac{\Delta t}{\Delta x}\Big( \frac{f^k_{i-\frac{1}{2}}}{\varepsilon} - \frac{f^k_{i+\frac{1}{2}}}{\varepsilon} \Big),
\end{equation}
\begin{equation}\label{n.5.21.11}
S^{k+1}_i =S^k_i + \frac{\Delta t}{\Delta x} \Big(  \frac{F^{k+1}_{i-\frac{1}{2}}}{\varepsilon} - \frac{F^{k+1}_{i+\frac{1}{2}}}{\varepsilon} \Big) + \Delta t n^{k+1}_i.
\vspace{0.25cm}
\end{equation}
But the expressions of $f^k_{i-\frac{1}{2}}$ and $F^{k+1}_{i-\frac{1}{2}}$, in Eqs. \eqref{n.5.12.1} and \eqref{n.5.12.2}, implies
\begin{eqnarray*}\label{n.5.22}
\frac{f_{i-\frac{1}{2}}^k}{\varepsilon} = \frac{2a^k_{i-\frac{1}{2}}D_n}{\varepsilon a^k_{i-\frac{1}{2}} \big( 1+\exp({-a^k_{i-\frac{1}{2}} \Delta x}) \big)- \big( \exp(-a^k_{i-\frac{1}{2}} \Delta x)-1 \big)}\\
\times \Big[ \frac{1}{2}(n^k_{i-1} + \varepsilon (nu)^k_{i-1}) -\frac{\exp(a^k_{i-\frac{1}{2}} \Delta x)}{2}( n^k_i - \varepsilon (nu)_i^k) \Big]
\end{eqnarray*}

\begin{eqnarray*}\label{n.5.23}
\text{and} \quad \frac{F^{k+1}_{i-\frac{1}{2}}}{\varepsilon} = \frac{D_S}{2\varepsilon D_S + \Delta x} \big( S^{k+1}_{i-1} + \varepsilon (SU_1)^{k+1}_{i-1} - S^k_i + \varepsilon (SU_1)^{k+1}_i \big).
\vspace{0.25cm}
\end{eqnarray*}
It follows that for every $i=0,\cdots,N_x$
\begin{equation}\label{n.5.24}
\lim_{\varepsilon \rightarrow 0^+} \frac{f_{i-\frac{1}{2}}^k}{\varepsilon} = \frac{\alpha_1 \partial_x^{(c)}S^k_{i-1}}{1-\exp(-\frac{\alpha_1 \Delta x}{D_n} \partial_x^{(c)}S^k_{i-1})}\big( n^k_{i-1} - \exp(-\frac{\alpha_1 \Delta x}{D_n} \partial_x^{(c)}S^k_{i-1})n^k_i \big),
\end{equation}
\begin{equation}\label{n.5.25}
\text{and}, \quad \lim_{\varepsilon \rightarrow 0^+}F^{k+1}_{i-\frac{1}{2}} = \frac{D_S}{\Delta x}(S^{k+1}_{i-1} - S^{k+1}_i).
\end{equation}
Passing to the limit, in \eqref{n.5.21.1}-\eqref{n.5.21.11}, and using the relations \eqref{n.5.24}-\eqref{n.5.25} yields the discretization \eqref{n.5.17}, \eqref{n.5.19}, \eqref{n.5.20}.
\\
\end{proof}
\subsection{Two dimensional numerical method}
In this section we will solve numerically the model \eqref{mod.n} in the two dimensional case.
Since the extension of the techniques proposed in previous section to higher dimension is still
not complete, we choose a discretization based on the Lax-Friedrichs scheme \cite{Elena-Vazquez,filbert1,LeVeque}.

Following the idea of \cite{filbert1} we write \eqref{mod.n} in the following form
\begin{equation}\label{n.6.1}
\displaystyle
\partial_t U+\partial_x F_1(U)+\partial_y F_2(U)=R(U),
\end{equation}
where
\begin{equation}\label{n.6.2}
\displaystyle
F_1(U)= \left(\;\;\begin{matrix}
nu_1\\
\frac{s^2 n}{2} \\
0 \\
N_1U_1^1 \\
\frac{s^2 N_1}{2} \\
0
\end{matrix}\;\;\right),
\quad
\displaystyle
F_2(U)= \left(\;\;\begin{matrix}
nu_2\\
 0\\
\frac{s^2 n}{2} \\
N_1U_1^2\\
0\\
\frac{s^2 N_1}{2}
\end{matrix}\;\;\right),
\end{equation}
\begin{equation}\label{n.6.3}
\quad
\text{and}
\quad
\displaystyle
R(U)= \left(\;\;\begin{matrix}
0\\
-\mu_1 n u+ \mu_2n\alpha_1(N_1)\nabla N_1\\
n\\
-\sigma_1N_1U_1
\end{matrix}\;\;\right),
\end{equation}
with
\begin{equation}\label{n.6.4}
u=(u_1,u_2),\qquad U_1=(U_1^1,U_1^2)\quad \text{and} \quad U=\left(\begin{matrix}
n\\ nu\\ N_1\\ N_1U_1
\end{matrix}\right).
\vspace{0.25cm}
\end{equation}

We use a Cartesian discretization of the rectangular domain $[-L_x,L_x]\times [-L_y,L_y]$
with steps $\Delta x$ and $\Delta y$.
The nodes of the mesh are denoted $(x_i,y_j)$ with $x_i=-L_x+i\Delta x$, $y_j=-L_y+j\Delta y$,
for $i=0,\ldots,N_x$ and $j=0,\ldots,N_y$. The time step is denoted $\Delta t$ and $t^k=k\Delta t$,
for $k\in\mathbb{N}$.

For each time step the equation \eqref{n.6.1} is solved using a time splitting method where the approximation $U^{k+1}_{i,j}$ is updated from $U^k_{i,j}$ in two steps: first we approximate the solution of equation \eqref{mod.n} without the source term ($R=0$), using the following scheme
\begin{equation}\label{n.6.5}
\displaystyle
U^{k+\frac{1}{2}}_{i,j}=U^k_{i,j}-\frac{\Delta t}{\Delta x} \big( F^{k+\frac{1}{2}}_{1,i+\frac{1}{2},j} - F^{k+\frac{1}{2}}_{1,i-\frac{1}{2},j} \big) - \frac{\Delta t}{\Delta y} \big( F^{k+\frac{1}{2}}_{2,i,j+\frac{1}{2}} - F^{k+\frac{1}{2}}_{2,i,j-\frac{1}{2}} \big),
\end{equation}
where the numerical flux $F^{k+\frac{1}{2}}_{1,i+\frac{1}{2},j}$,  $F^{k+\frac{1}{2}}_{1,i-\frac{1}{2},j}$, $F^{k+\frac{1}{2}}_{2,i,j+\frac{1}{2}},$ and $F^{k+\frac{1}{2}}_{2,i,j-\frac{1}{2}}$ are given by the Lax-Friedrichs flux \cite{Elena-Vazquez}
\begin{equation*}\label{n.6.6}
F^{k+\frac{1}{2}}_{1,i+\frac{1}{2},j} = \frac{1}{2} \big( F_1\big( U^{k+\frac{1}{2}}_{i,j} \big) + F_1 \big( U^{k+\frac{1}{2}}_{i+1,j} \big) \big) - \frac{\alpha_x}{2}\big( U^{k+\frac{1}{2}}_{i+1,j} -U^{k+\frac{1}{2}}_{i,j} \big),
\end{equation*}
\begin{equation*}\label{n.6.7}
F^{k+\frac{1}{2}}_{1,i-\frac{1}{2},j} = \frac{1}{2} \big( F_1\big( U^{k+\frac{1}{2}}_{i-1,j} \big) + F_1 \big( U^{k+\frac{1}{2}}_{i,j} \big) \big) - \frac{\alpha_x}{2}\big( U^{k+\frac{1}{2}}_{i,j} -U^{k+\frac{1}{2}}_{i-1,j} \big),
\end{equation*}
\begin{equation*}\label{n.6.8}
F^{k+\frac{1}{2}}_{2,i,j+\frac{1}{2}} = \frac{1}{2} \big( F_2\big( U^{k+\frac{1}{2}}_{i,j} \big) + F_2 \big( U^{k+\frac{1}{2}}_{i,j+1} \big) \big) - \frac{\alpha_y}{2}\big( U^{k+\frac{1}{2}}_{i,j+1} -U^{k+\frac{1}{2}}_{i,j} \big),
\end{equation*}
\begin{equation*}\label{n.6.9}
F^{k+\frac{1}{2}}_{2,i,j-\frac{1}{2}} = \frac{1}{2} \big( F_2\big( U^{k+\frac{1}{2}}_{i,j-1} \big) + F_2 \big( U^{k+\frac{1}{2}}_{i,j} \big) \big) - \frac{\alpha_y}{2}\big( U^{k+\frac{1}{2}}_{i,j} -U^{k+\frac{1}{2}}_{i,j-1} \big).
\end{equation*}
\begin{figure}
\hspace*{-1.2cm}
\begin{subfigure}{.58\textwidth}
  \centering
  \includegraphics[width=1.\linewidth]{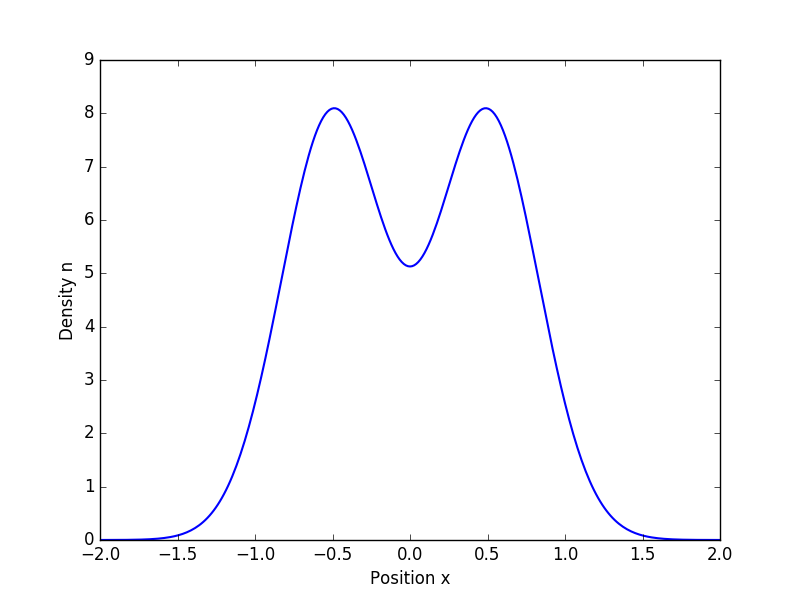}
  \caption{Density at $t=0.01$}
  \label{fig:sfig1}
\end{subfigure}%
\begin{subfigure}{.58\textwidth}
  \centering
  \includegraphics[width=1.\linewidth]{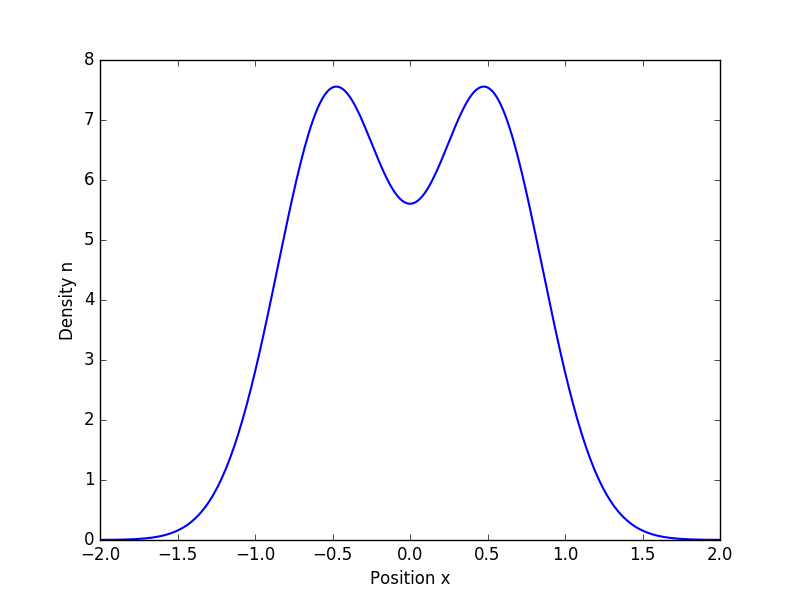}
  \caption{Density at $t=0.02$}
  \label{fig:sfig2}
\end{subfigure}
 \hspace*{-1.2cm}
 \begin{subfigure}{.58\textwidth}
  \centering
  \includegraphics[width=1.\linewidth]{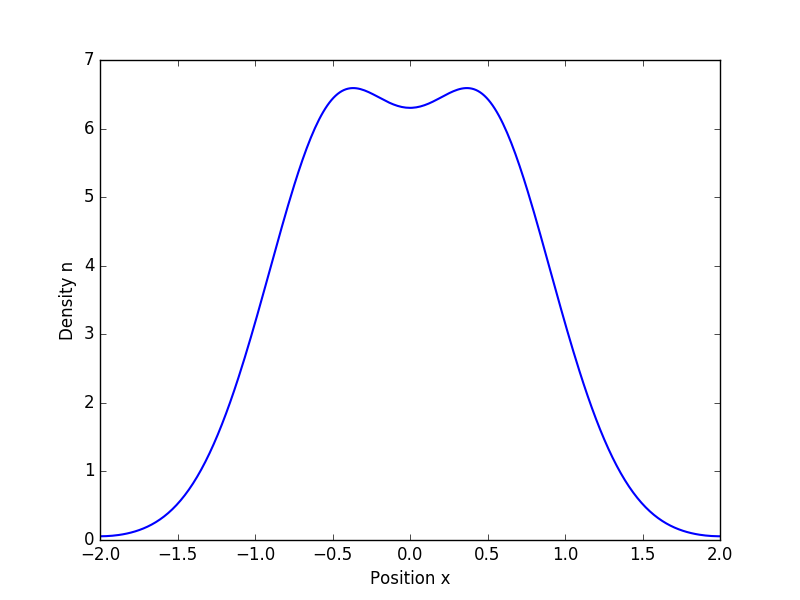}
  \caption{Density at $t=0.06$}
  \label{fig:sfig2}
\end{subfigure}%
\begin{subfigure}{.58\textwidth}
  \centering
  \includegraphics[width=1.\linewidth]{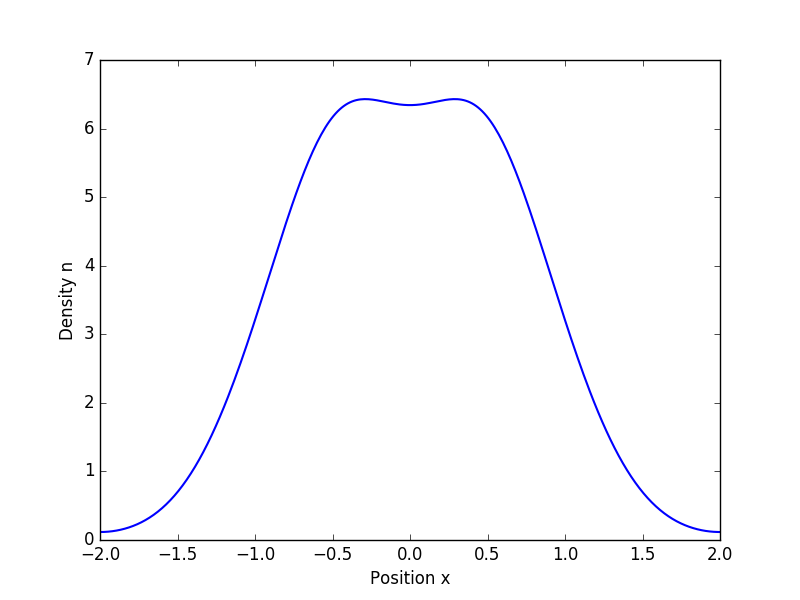}
  \caption{Density at $t=0.08$}
  \label{fig:sfig2}
\end{subfigure}
\caption{Time dynamics of the cell density $n(t,x)$ obtained from the WB scheme with $s=5^9$ on the domain $[-2,2]$. Parameter values: $\alpha_1=0.33$, $D_n=1$, $D_S=0.001$.}
\label{fig:fig}
\end{figure}

\begin{figure}
  \hspace*{-1.2cm}
  \begin{subfigure}{.58\textwidth}
  \centering
  \includegraphics[width=1.\linewidth]{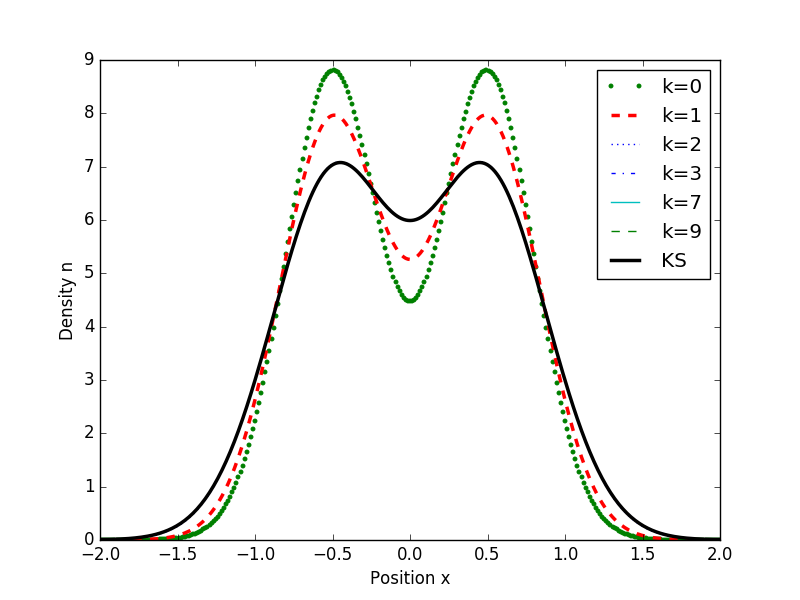}
  \caption{Density at $t=0.03$}
  \label{fig:sfig1}
\end{subfigure}%
\begin{subfigure}{.58\textwidth}
  \centering
  \includegraphics[width=1.\linewidth]{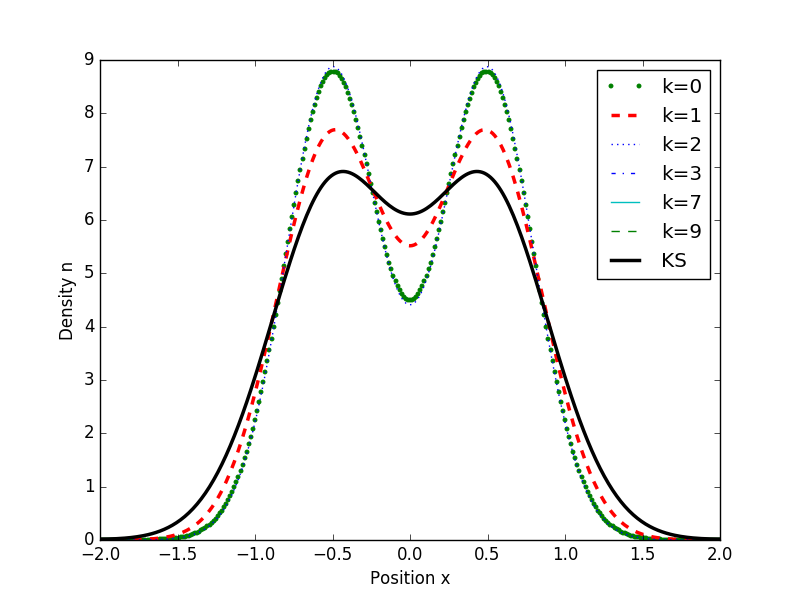}
  \caption{Density at $t=0.04$}
  \label{fig:sfig2}
\end{subfigure}
  \hspace*{-1.2cm}
  \begin{subfigure}{.58\textwidth}
  \centering
  \includegraphics[width=1.\linewidth]{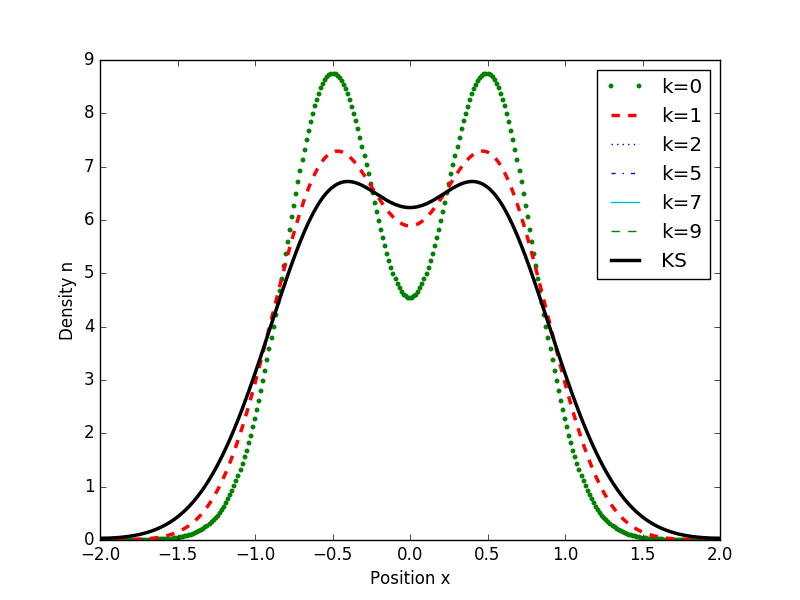}
  \caption{Density at $t=0.05$}
  \label{fig:sfig2}
\end{subfigure}%
\begin{subfigure}{.58\textwidth}
  \centering
  \includegraphics[width=1.\linewidth]{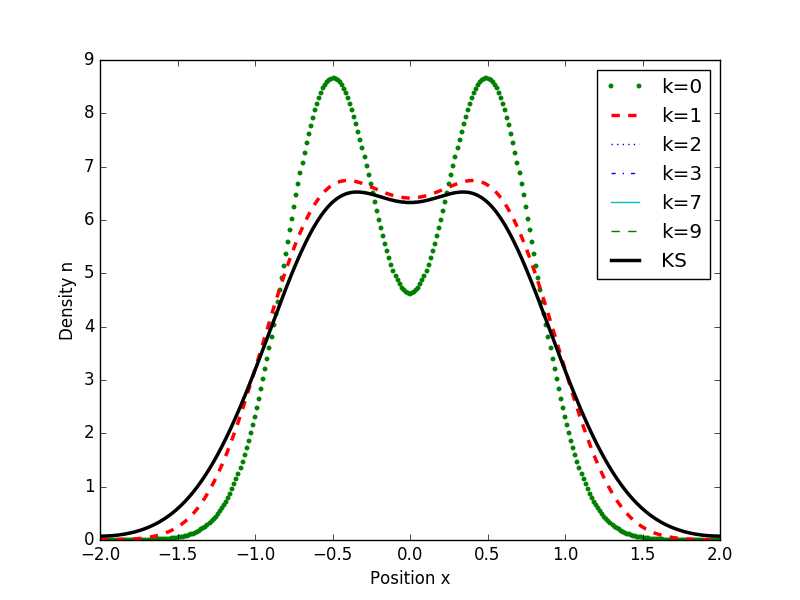}
  \caption{Density at $t=0.07$}
  \label{fig:sfig2}
\end{subfigure}
\caption{Time dynamics of the cell density $n(t,x)$ obtained from the WB scheme with $s=5^k$, $k=0,1,2,5,7,9$ and comparison with KS on the domain $[-2,2]$. Parameter values: $\alpha_1=0.33$, $D_n=1$, $D_S=0.001$.}
\label{fig:fig}
\end{figure}

\begin{figure}
\begin{center}
\hspace*{-1.5cm}
\begin{subfigure}{.6\textwidth}
 \centering
  \includegraphics[width=1.\linewidth]{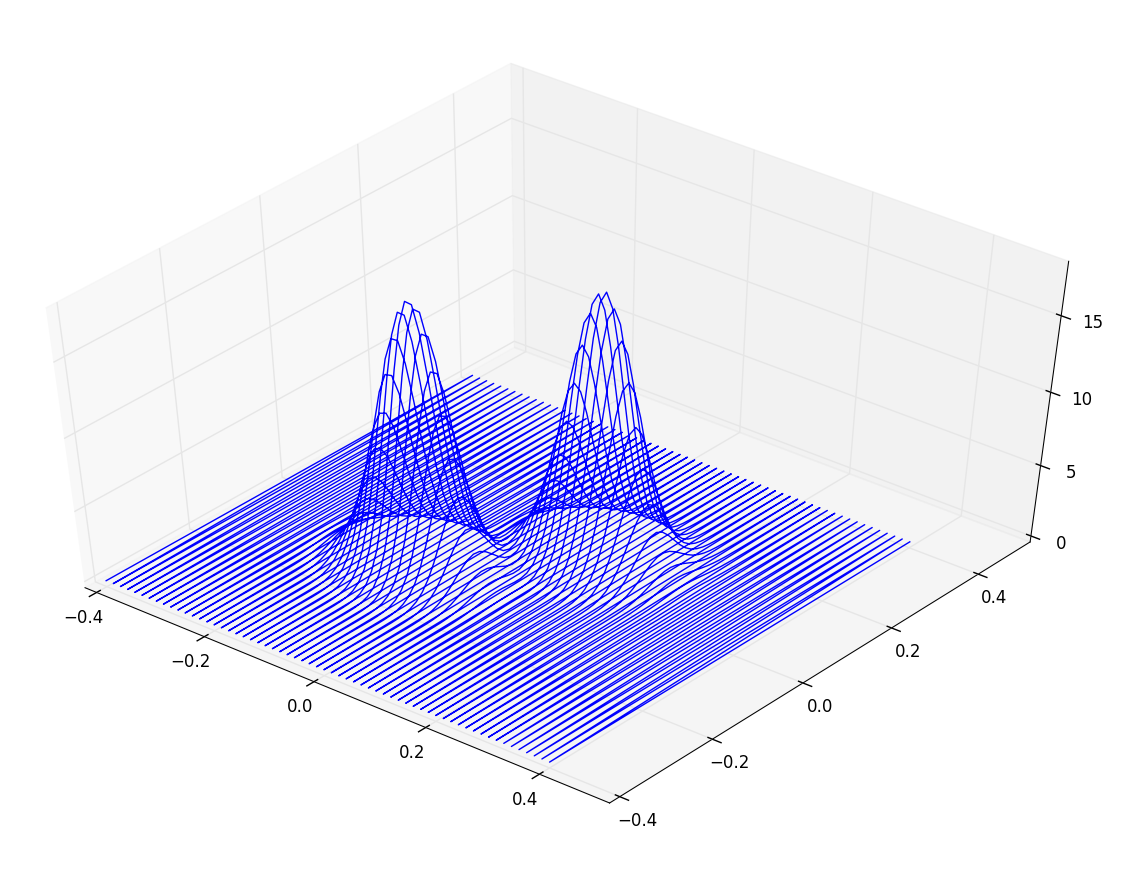}
  \caption{Density at $t=0.001$}
  \label{fig:sfig1}
\end{subfigure}%
\begin{subfigure}{.6\textwidth}
  \centering
  \includegraphics[width=1.\linewidth]{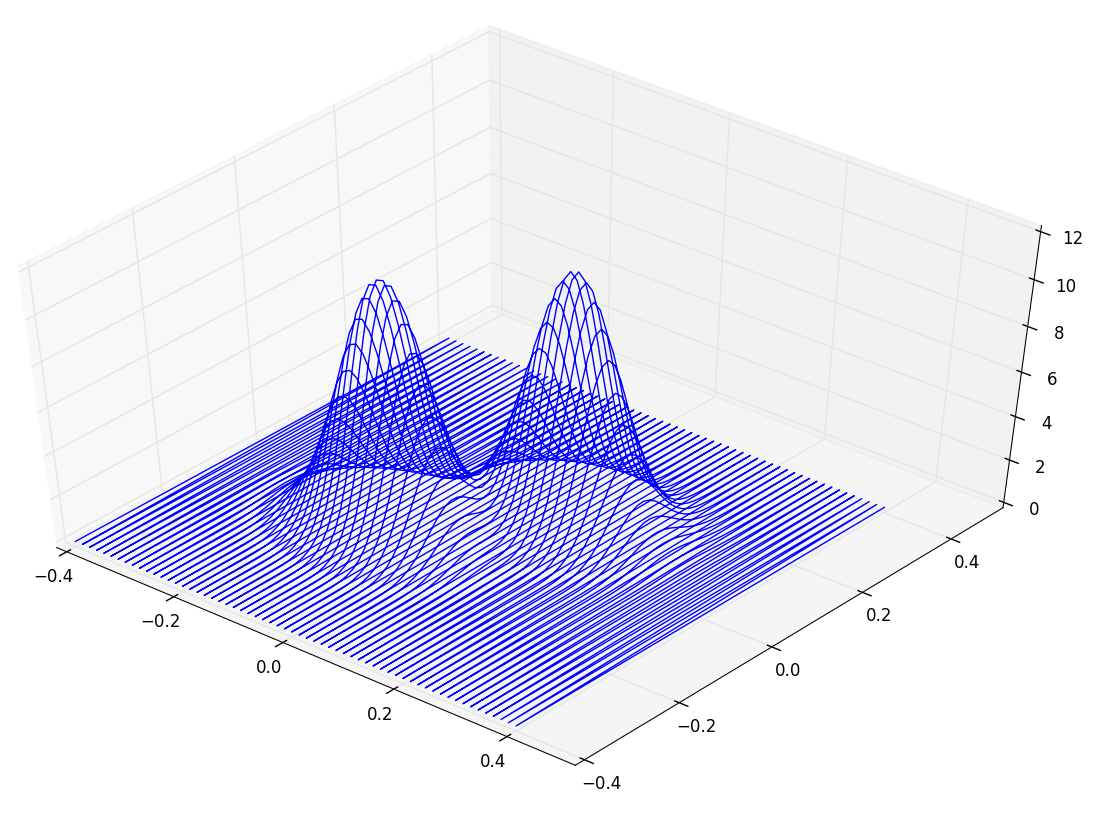}
  \caption{Density at $t=0.002$}
  \label{fig:sfig2}
\end{subfigure}%

\begin{subfigure}{.6\textwidth}
  \centering
  \includegraphics[width=1.\linewidth]{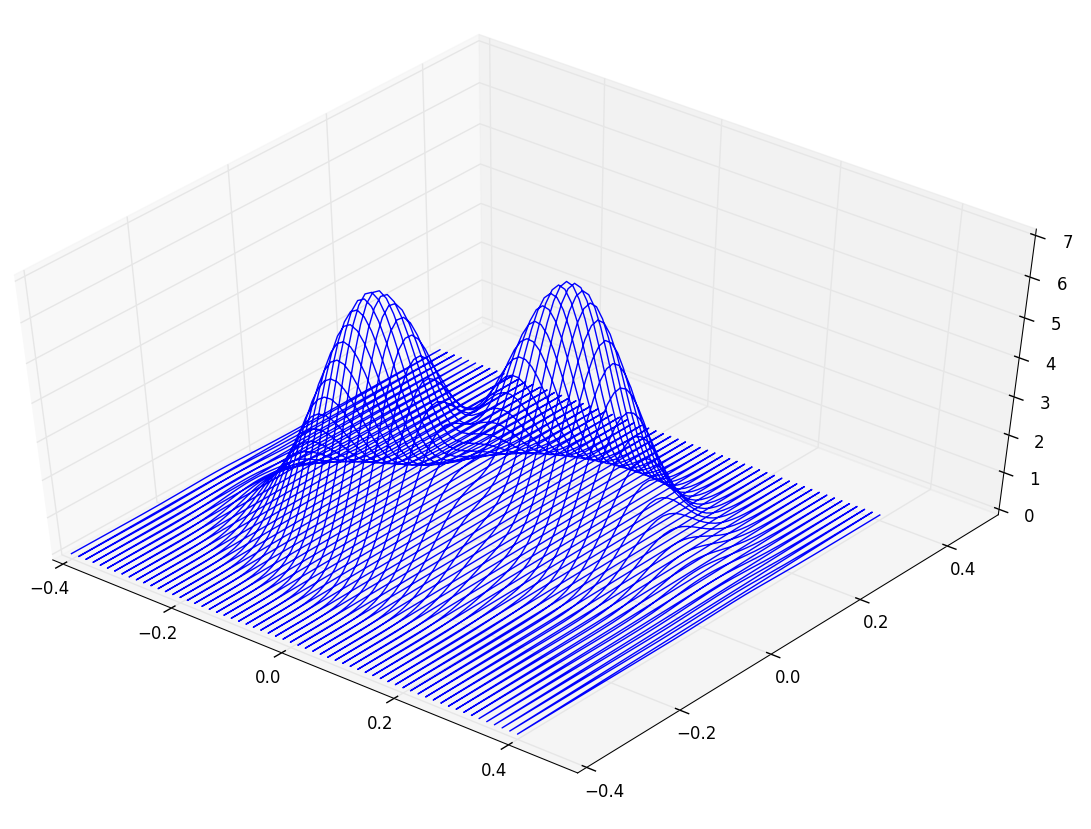}
  \caption{Density at $t=0.004$}
  \label{fig:sfig2}
\end{subfigure}%
\caption{Time dynamics of the cell density $n(t,x)$ obtained from the two-dimensional scheme LF with $s=100$ on a square domaine $[-0.4,0.4] \times [-0.4,0.4]$. Parameter values: $\alpha_1=0.33$, $D_n=1$, $D_S=0.001$.}
\label{fig:fig}
\end{center}
\end{figure}
Here the contants $\alpha_x$ and $\alpha_y$ are defined by
\begin{equation*}\label{n.6.10}
\alpha_x = \max_{k=1,\cdots,6} \{|\lambda_k^1|\}, \quad \text{and} \quad \alpha_y = \max_{k=1,\cdots,6} \{|\lambda_k^2|\},
\end{equation*}
where $\lambda_k^1$ (respectively $\lambda_k^2$) is the eigenvalue of the the jacobian matrix $F'_1(U)$ (respectively $F'_2(U))$.\\

Next, the approximation $U^{k+1}_{i,j}$ is computed from the approximation $U^{k+\frac{1}{2}}_{i,j}$ by
\begin{equation}\label{n.6.6.10}
U^{k+1}_{i,j} = U^{k+\frac{1}{2}}_{i,j} + \Delta t R_d(U^{k+1}_{i,j}),
\end{equation}
with
\begin{equation}\label{n.6.11}
\displaystyle
R_d(U^{k+1}_{i,j})= \left(\;\;\begin{matrix}
0\\
\\
-\mu_1 (n u_1)^{k+1}_{i,j}+ \mu_2n^{k+1}_{i,j}\alpha_1\frac{N^{k+1}_{1,i+1,j}-N^{k+1}_{1,i-1,j}}{2 \Delta x}\\
\\
-\mu_1 (n u_2)^{k+1}_{i,j}+ \mu_2n^{k+1}_{i,j}\alpha_1\frac{N^{k+1}_{1,i,+1j}-N^{k+1}_{1,i,-1j}}{2 \Delta y}\\
\\
n^{k+1}_{i,j}\\
\\
-\sigma_1(N_1U_1)^{k+1}_{i,j}
\end{matrix}\;\;\right),
\vspace{0.25cm}
\end{equation}
As in the one space dimensional case we complete the system with Neumann boundary conditions
\begin{equation}
\nabla n\cdot \eta \big|_{\partial \Omega} = 0, \quad \text{and} \quad \nabla N_1\cdot \eta \big|_{\partial \Omega} = 0
\end{equation}
for the density $n$ and for the concentration $N_1$ and we impose Dirichlet boundary conditions for the flux $q:=nu$ and $q_1:=N_1U_1$:
\begin{equation}\label{n.6.13}
q\cdot \eta \big|_{\partial \Omega} = 0, \quad \text{and} \quad q_1\cdot \eta \big|_{\partial \Omega} = 0.
\end{equation}

\subsection{Numerical tests}
We present here some numerical experiments in both cases: in one space dimension and in the two dimensional case. For all numerical tests carried out below, we take
\begin{equation*}
\alpha_1(N_1)=0.33, \quad D_n=1, \quad \text{and} \quad D_S=0.001.
\end{equation*}
For the initial conditions we consider an initial datum for the chemical concentration $S$ (=$N_1$) and for the flux $nu$ which are at rest,
\begin{equation*}
S(0)=0, \quad \text{and} \quad (nu)(0)=0.
\end{equation*}
Concerning the density of cells $n$, we take
\begin{equation*}
n(0,x)=\frac{n_0}{2\pi \sigma^2} \Big( \exp \big(-\frac{(x-x_0)^2}{2\sigma^2} \big) + \exp \big(-\frac{(x+x_0)^2}{2\sigma^2} \big) \Big),
\end{equation*}
in one space dimension, where $n_0=5,$ $x_0=0.5$ and $\sigma=3.10^{-1}$. In two space dimension we consider \cite{filbert1}
\begin{equation*}
n(0,x,y)=\frac{n_0}{2\pi\sigma^2} \Big( \exp \big(-\frac{(x-x_0)^2 + (y-y_0)^2}{2\sigma^2} \big) + \exp \big(-\frac{(x+x_0)^2 + (y+y_0)^2}{2\sigma^2} \big) \Big),
\end{equation*}
where $n_0=0.25$, $(x_0,y_0)=(3\sigma,3\sigma)$ and $\sigma=3.10^{-2}$.\\

In the following, we denote by
\begin{enumerate}
\item[$\bullet$] WB: the well-balanced asymptotic preserving scheme \eqref{n.5.14};\\

\item[$\bullet$] KS: the scheme \eqref{n.5.17}, \eqref{n.5.20} for the Keller-Segel system;\\

\item[$\bullet$] LF: the Lax-Friedrichs scheme \eqref{n.6.5}-\eqref{n.6.6.10}.
\end{enumerate}

We illustrate in Figure 1. the behavior of the WB scheme at successive times ($t=0.01$, $0.02$, $0.06$, $0.08$). It can be seen that with the evolution of time we observe the union of the two initial high density regions of $n$. In Figure 2. we plot at successive times ($t=0.03$, $0.04$, $0.05$, $0.07$) the density of cells obtained from the WB scheme for different values of $\varepsilon$ ($\varepsilon=5^{-k}$, $k=0,1,2,3,7,9$). We also compare with the numerical result obtained with the KS scheme. Clearly the WB scheme converge as $\varepsilon \longrightarrow 0$ to the KS limit. It illustrate the result of in Proposition \ref{psp}.

The behavior of the model \eqref{mod.n} in the two-dimensional case is illustrated in the Figure 3. where we plot the density of cells obtained from LF scheme at different times ($t=0.001$, $0.002$, $0.004$). As in the one-dimensional case we observe the union of the two initial high density regions of $n$.



\begin{thebibliography}{99}

\bibitem{alt}
W.~Alt, Biased random walk models for chemotaxis and related diffusion approximations,
{\em J. Math. Biol.,} \textbf{9} 147-177, (1980).

\bibitem{BBC}
N.~Bellomo, A.~Bellouquid, N.~Chouhad, From a multiscale derivation of nonlinear cross-diffusion models to Keller-Segel models in a Navier-Stokes fluid,
{\em Math. Models Methods Appl. Sci.,} \textbf{26}, (2016), DOI: 10.1142/S0218202516400078.



\bibitem{bellomo8}
N.~Bellomo, A.~Bellouquid, J.~Nieto and J.~Soler, Multicellular biological growing systems: Hyperbolic limits towards macroscopic description, {\em Math. Models Methods Appl. Sci.,} \textbf{17} 1675-1692, (2007).

\bibitem{bellomo9}
N.~Bellomo, A.~Bellouquid, J.~Nieto and J.~Soler, Multiscale biological tissue models and flux-limited chemotaxis for multicellular growing systems, {\em Math. Models Methods in Appl. Sci.,} \textbf{20(7)} 1179-1207, (2010).

\bibitem{bellomo7}
N.~Bellomo, A.~Bellouquid, J.~Nieto and J.~Soler, On the asymptotic theory from microscopic to macroscopic growing tissue models: An overview with perspectives, {\em Math. Models Methods Appl. Sci.,} \textbf{22(1)} 1130001, (2012).

\bibitem{bellomo12}
N.~Bellomo, A.~Bellouquid, Y.~Tao and M.~Winkler, Toward a mathematical theory of Keller-Segel models of pattern formation in biological tissues, {\em Math. Models Methods Appl. Sci.,} \textbf{25(9)} 1663-1763, (2015).

\bibitem{bellouquid5}
A.~Bellouquid and E.~De Angelis, From kinetic models of multicellular growing systems to macroscopic biological tissue models, {\em Nonlinear Anal. Real World Appl.,} \textbf{12} 1111-1122, (2011).

\bibitem{BT}
Bingran Hu, Y.~Tao, To the exclusion of blow-up in a three-dimensional chemotaxis-growth model with indirect attractant production,
{\em Math. Models Methods Appl. Sci.,} \textbf{26}, (2016), DOI: 10.1142/S0218202516400091.
	

\bibitem{chalub}
F.A.~Chalub, P.A.~Markowich, B.~Perthame and C.~Schmeiser, Kinetic models for chemotaxis and their driftÂ diffusion limits, {\em Monatsh. Math.,} \textbf{142} 123-141, (2004).

\bibitem{corrias}
L.~Corrias, B.~Perthame and H.~Zaag, A chemotaxis model motivated by angiogenesis, {\em C. R. Acad. Sci. Paris, Ser.,} \textbf{I(336)} 141-146, (2003).

\bibitem{dolak}
Y.~Dolak and T.~Hillen, Cattaneo models for chemosensitive movement: Numerical solution and pattern formation, {\em J. Math. Biol.,} \textbf{46} 153-170, (2003).

\bibitem{dolak2}
Y.~Dolak and C.~Schmeiser, Kinetic models for chemotaxis: Hydrodynamic limits and spatio-temporal mechanisms, {\em J. Math. Biol.,} \textbf{51} 595-615, (2005).

\bibitem{Elena-Vazquez}
M.~Elena V\'azquez-Cend\'on,  \textbf{Solving Hyperbolic Equations with Finite Volume Methods}, Springer, (2015).

\bibitem{almeida}
C.~Emako-Kazianou, L. Neves de Almeida and N.~Vauchelet, Existence and diffusive limit of a two-species kinetic model of chemotaxis, {\em Kinet. Relat. Models,} \textbf{8(2)} 359-380, (2015).

\bibitem{filbert1}
F.~Filbet, P.~Lauren\c{c}ot and B.~Perthame, Derivation of hyperbolic models for chemosensitive movement, {\em J. Math. Biol.,} \textbf{50} 189-207, (2005).

\bibitem{goudon}
T.~Goudon, O.~S\'anchez, J.~Soler and L.L.~Bonilla, Low-field limit for a nonlinear discrete drift-diffusion model arising in semiconductor superlattices theory, {\em SIAM J. Appl. Math.,} \textbf{64(5)} 1526-1549, (2004).

\bibitem{Laurent-Gosse}
L.~Gosse, \textbf{Computing Qualitatively Correct Approximations of Balance Laws: Exponential-Fit, Well-Balanced and Asymptotic-Preserving}, Springer-Verlag Italia, (2013).

\bibitem{gosse-toscani}
L.~Gosse and G.~Toscani, An asymptotic-preserving well-balanced scheme for the hyperbolic heat equations, C. R. Acad. Sci. Paris, Ser., \textbf{I(334)} 337-342, (2002).

\bibitem{hwang}
H.J.~Hwang, K.~Kang and A.~Stevens, Drift-diffusion limits of kinetic models for chemotaxis: a generalization, {\em Disc. Cont. Dyn. Syst. Series B,} \textbf{5(2)} 319-334, (2005).

\bibitem{hillen}
T.~Hillen and H.G.~Othmer, The diffusion limit of transport equations derived from velocity-jump processes, {\em SIAM J. Appl. Math,}
\textbf{61(3)} 751-775, (2000).

\bibitem{hillen2}
T.~Hillen, Hyperbolic models for chemosensitive movement, {\em Math. Models Methods App. Sci.,} \textbf{12(7)} 1007-1034, (2002).

\bibitem{hillen4}
T.~Hillen, On the $L^{2}$-moment closure of transport equations: The Cattaneo approximation, {\em Disc. Cont. Dyn. Syst. Series B,} \textbf{4(4)} 961-982, (2004).

\bibitem{james}
F.~James and N.~Vauchelet, Chemotaxis: From kinetic equations to aggregate dynamics, {\em Nonlinear Differ. Equ. Appl.,} \textbf{20(1)} 101-127, (2013).

\bibitem{keller}
E.F.~Keller and L.A.~Segel, Traveling bands of chemotactic bacteria: A theoretical analysis, {\em J. Theor. Biol.,} \textbf{30} 235-248, (1971).

\bibitem{L}
J.~Lankeit, Long-term behaviour in a chemotaxis fluid system with logistic source,
{\em Math. Models Methods Appl. Sci.,} \textbf{26}, (2016), DOI: 10.1142/S021820251640008X.

\bibitem{LeVeque}
R.J.~LeVeque, \textbf{Numerical Methods for Conservation Laws}, Second edition, Birkh\"{a}user Verlag, (1992).

\bibitem{othmer} 
H.G.~Othmer, S.R.~Dunbar and W.~Alt, Models of dispersal in biological systems, {\em J. Math. Biol.,} \textbf{26} 263-298, (1988).

\bibitem{othmer1} 
H.G.~Othmer and T.~Hillen, The diffusion limit of transport equations II: Chemotaxis equations, {\em SIAM J. Appl. Math.,} \textbf{62(4)} 1222-1250, (2002).

\bibitem{patlak}
C.S.~Patlak, Random walk with persistence and external bias, {\em Bull. Math. Bio-phys.,} \textbf{15(3)} 311-338, (1953).

\bibitem{perthame}
B.~Perthame, PDE models for chemotactic movements: Parabolic, hyperbolic and kinetic, {\em Appl. Math.,} \textbf{49(6)} 539-564, (2004).


\bibitem{ScharfGum}
D.L. Scharfetter and H.K. Gummel, Large signal analysis of a silicon read diode oscillator, {\em IEEE Trans. Electron Dev.,} \textbf{16(1)} 64-77, (1969).



\bibitem{stevens}
A.~Stevens, The derivation of chemotaxis equations as limit dynamics of moderately interacting stochastic many-particle systems, {\em SIAM J. Appl. Math.,} \textbf{61(1)} 183-212, (2000).

\bibitem{SSU}
C.~Stinner, C.~Surulescu, and  A.~Uatay, Global existence of a go-or-grow multiscale model for tumor invasion with therapy, {\em Math. Models Methods Appl. Sci.,} \textbf{26}, (2016),DOI: 10.1142/S021820251640011X.

\bibitem{TW}
J.~I.~Tello, D.~Wrzosek, Predator-prey model with diffusion and indirect prey-taxis,
{\em Math. Models Methods Appl. Sci.,} \textbf{26}, (2016), DOI: 10.1142/S0218202516400108.


\bibitem{vauchelet}
N.~Vauchelet, Numerical simulation of a kinetic model for chemotaxis, {\em Kinet. Relat. Models,} \textbf{3(3)} 501-528, (2010).

\bibitem{villani}
C.~Villani, Trend to equilibrium for dissipative equations, functional inequalities and mass transportation, in M.C.~Carvalho and J.F.~Rodrigues (eds.), \textbf{Recent Advances in the Theory and Applications of Mass Transport}, pp. 95-109, {\em Contemp. Math.,} \textbf{353}, Amer. Math. Soc., Providence, RI, (2004).

\bibitem{W}
M.~Winkler, Chemotactic cross-diffusion in complex frameworks, {\em Math. Models Methods Appl. Sci.,} \textbf{26}, (2016), DOI: 10.1142/S0218202516020024.

\bibitem{WIN}
M.~Winkler, The two-dimensional KellerSegel system with singular sensitivity and signal absorption: Global large-data solutions and their relaxation properties, {\em Math. Models Methods Appl. Sci.,} \textbf{26} 987-1024, (2016).

\end{thebibliography}
\end{document}